\documentclass[12pt, reqno]{amsart}
\usepackage{amsmath, amsthm, amscd, amsfonts, amssymb, graphicx, color}
\usepackage[bookmarksnumbered, colorlinks, plainpages]{hyperref}
\usepackage{cite}
\newtheorem*{theoA}{Theorem A}
\newtheorem*{theoB}{Theorem B}
\newtheorem*{theoC}{Theorem C}
\newtheorem*{theoD}{Theorem D}

\newtheorem{theo}{Theorem}[section]
\newtheorem{lem}{Lemma}[section]

\newtheorem{exm}{Example}[section]

\newtheorem{ques}{Question}[section]
\newcommand{\ol}{\overline}
\newcommand{\be}{\begin{equation}}
	\newcommand{\ee}{\end{equation}}
\newcommand{\beas}{\begin{eqnarray*}}
	\newcommand{\eeas}{\end{eqnarray*}}
\newcommand{\bea}{\begin{eqnarray}}
	\newcommand{\eea}{\end{eqnarray}}
\numberwithin{equation}{section}

\begin{document}
	\setcounter{page}{1}
	
\title[existence of entire solutions.....]{On Entire solutions of system of Fermat-type difference and partial differential-difference equations in $\mathbb{C}^n$}
	
	\author[G. Haldar]{Goutam Haldar}
	
	\address{Department of Mathematics, Malda College, Malda, West Bengal 732101, India.}
	\email{goutamiit1986@gmail.com}
	
	
	
	\date{Received: xxxxxx; Revised: yyyyyy; Accepted: zzzzzz.
		\newline\indent $^{*}$ Corresponding author}
	
\begin{abstract} 
The equation $f^n+g^n=1$, $n\in\mathbb{N}$ can be regarded as the Fermat Diophantine equation over the function field. In this paper we study the characterization of entire solutions of some system of Fermat type functional equations by taking $e^{g_1(z)}$ and $e^{g_2(z)}$ in the right hand side of each equation, where $g_1(z)$ and $g_2(z)$ are polynomials in $\mathbb{C}^n$. Our results extend and generalize some recent results. Moreover, some examples have been exhibited to show that our results are precise to some extent.
\end{abstract} \maketitle
	
\section{Introduction and Results}
It is well known that Nevanlinna theory is an important tool to study value distribution of entire and meromorphic solutions on complex differential equations (see \cite{Hayman & 1964, Laine & 1993}). The study of finding the entire and meromorphic solutions with precise form of different variants Fermat type functional equation is an interesting topic in the field of complex analysis. The Fermat type functional equation has been originated from  the Fermat type equation $x^n+y^m=1$, where $m,n\in \mathbb{N}$, which were studied by Wiles and Taylor \cite{Tailor & Wiles & 1995, Wiles & Ann. Math. 1995}, around 1995, and they pointed out that it does not admit nontrivial solution in rational numbers for $m=n\geq3$, and does exist nontrivial solution in rational numbers for $m=n=2$. Initially, Fermat-type functional equations were investigated by Montel \cite{Montel & Paris & 1927}, Gross \cite{Gross & Bull. Amer. & 1966, Gross & Amer. Math. & 1966} and Iyer \cite{Iyer & J. Indian. Math. Soc. & 1939}, independently.\vspace{1.2mm}
\par Recently, the difference analogues of Nevanlinna theory, specially the development of difference analogous lemma of the logarithmic derivative has been established by Halburd and Korhonen \cite{Halburd & Korhonen & 2006, Halburd & Korhonen & Ann. Acad. & 2006}, and Chiang and Feng \cite{Chiang & Feng & 2008}, independently. Since then many researchers have started to study the existence of entire and meromorphic solutions of complex difference as well as complex differential-difference equations, and obtained a number of important and interesting results in the literature (see \cite{Liu-Gao & Sci. Math. & 2019,Liu & Cao & Arch. Math. & 2012,Liu & Yang & 2016,Tang & Liao & 2007}).\vspace{1.2mm}
\par In 2012, Liu \textit{et al.} \cite{Liu & Cao & Arch. Math. & 2012} proved that the entire solutions of the Fermat type difference equation $f(z)^2+f(z+c)^2=1$ must be of the form $f(z)=\sin(Az+B)$, where $c(\neq0)$, $B\in\mathbb{C}$ and $A=(4k+1)\pi/2c$, $k$ is an integer. Later in $ 2013 $, Liu and Yang \cite{Liu & Yang & CMFT & 2013} extended this result by considering the Fermat-type difference equation $f(z)^2+P(z)^2f(z+c)^2=Q(z)$, where $P(z)$ and $Q(z)$ are two non-zero polynomials. In the same paper they have also proved that the finite order transcendental entire solutions of $f^{\prime}(z)^2+f(z+c)^2=1$ must assumes the form $f(z)=\sin(z\pm Bi)$, where $B$ is a constant and $c=2k\pi$ or $c=(2k+1)\pi$, $k$ is an integer.\vspace{1.2mm}
\par To the best of our knowledge, Gao \cite{Gao & Acta Math. Sinica & 2016}, in 2016, first studied about the entire solutions of system of differential difference equations and obtained the result as follows. If $(f_1,f_2)$ be a pair of finite order transcendental entire solution of the system of differential-difference equations \bea\label{e1.5}\begin{cases}
f_1^{\prime}(z)^2+f_2(z+c)^2=1\\f_2^{\prime}(z)^2+f_1(z+c)^2=1,\end{cases}\eea then $(f_1,f_2)$ must be of the form $(f_1,f_2)=(\sin(z-ib),\sin(z-ib_1)),\;\;\text{or}\;\; (\sin(z+ib),\sin(z+ib_1)),$ where $b,b_1$ are constants and $c=k\pi$, $k$ is a integer. In 2022, Wu-Hu \cite{Wu Hu & Analysis Math & 2022} characterized transcendental entire solutions for a system of Fermat type q-differential-difference equations which generalizes
the result of Gao \cite{Gao & Acta Math. Sinica & 2016}.\vspace{1.2mm}
\section{Solutions of Fermat-type differential-difference equations in several complex variables}
\par It has a long history to study partial differential equations which is a generalizations of the
well-known eikonal equation in real variable case. We insist the readers to go through \cite{Courtant Hilbert & 1962,Garabedian & Wiley & 1964,Rauch & 1991} and the references therein. Recently, the study of entire and meromorphic solutions of Fermat type partial differential equations in several complex variables has received considerable attention in the literature (see \cite{Li & Arch. Math. & 2008,Li & Int. J. & 2004,Saleeby & Analysis & 1999,Saleeby & Complex Var. & 2004,Li Complex Var & 1996}). In 1995, Khavinson pointed out that in $\mathbb{C}^2$, the entire solution of the Fermat type partial differential equations $f_{z_1}^2+f_{z_2}^2=1$ must necessarily be linear. After the development of the difference Nevanlinna theory in several
complex variables, specially the difference version of logarithmic derivative lemma (see \cite{Biancofiore & Stoll & 1981,Cao & Xu & Ann. Math. Pure Appl. & 2020,Korhonen & CMFT & 2012}), many mathematicians have started to study the existence and the precise form of entire and meromorphic solutions of different variants of Fermat type difference and partial differential difference equations, and obtained very remarkable and interesting results (see \cite{Xu & Cao & 2018,Xu & Cao & 2020,Xu & wang & 2020,Zheng-Xu & Analysis math & 2021,Haldar & 2023 & Mediterr,Haldar-Ahamed & Anal. Math. Phys. 2022,Xu Haldar & EJDE & 2023}).\vspace{1.2mm}
\par Hereinafter, we denote by $z+w=(z_1+w_1,z_2+w_2,\ldots,z_n+w_n)$, where $z=(z_1,z_2,\ldots,z_n),w=(w_1,w_2,\ldots,w_n)\in\mathbb{C}^n$, $n$ being a positive integer.\vspace{1.2mm}
\par In 2018, Xu-Cao \cite{Xu & Cao & 2018} proved the following results which are the extensions and generalizations of some previous results given by Liu \textit{et al.} \cite{Liu & Cao & Arch. Math. & 2012} from single complex variable to several complex variables. We list some of them here.\vspace{1.2mm} 
\begin{theoA}{\em\cite{Xu & Cao & 2018}}
Let $c=(c_1,c_2,\ldots,c_n)\in \mathbb{C}^n\setminus\{(0,0,\ldots,0)\}$. If $f:\mathbb{C}^n\rightarrow\mathbb{P}^{1}(\mathbb{C})$ be an entire solution with finite order of the Fermat type difference equation \bea\label{e2.1a} f(z)^2+f(z+c)^2=1,\eea the $f(z)$must assume the form $f(z)=\cos(L(z)+B)$, where $L$ is a linear function of the form $L(z)=a_1z_1+\cdots+ a_nz_n$ on $\mathbb{C}^n$ such that $L(c)=-\pi/2-2k\pi$ $(k\in\mathbb{Z})$, and $B$ is a constant on $\mathbb{C}$.
\end{theoA}
\begin{theoB}{\em\cite{Xu & Cao & 2018}}
Let $c=(c_1,c_2)\in\mathbb{C}^2$. If $f(z)$ be transcendental entire solution with finite order of the Fermat type partial differential-difference equation \bea\label{e2.2a} \left(\frac{\partial f(z_1,z_2)}{\partial z_1}\right)^2+f(z_1+c_1,z_2+c_2)^2=1,\eea $f$ must be of the form  $f(z_1,z_2)=\sin(Az_1+Bz_2+\phi(z_2))$, Where $A, B$ are constant on $\mathbb{C}$ satisfying $A^2=1$ and $Ae^{i(Ac_1+Bc_2)}=1$, and $\phi(z_2)$ is a polynomial in one variable $z_2$ such that $\phi(z_2)\equiv \phi(z_2 + c_2)$. In the special case whenever $c_2\neq0$, we have $f(z_1,z_2)=\sin (Az_1+Bz_2+ Constant)$.\end{theoB}
\par Motivated be Theorems A and B, Xu \textit{et al.}, in 2020 \cite{Xu-Liu-Li-JMAA-2020}, studied about entire solutions for the system of the Fermat type difference as well as partial differential difference equations in $\mathbb{C}^2$ and proved the following two results.
\begin{theoC}{\em\cite{Xu-Liu-Li-JMAA-2020}}
Let $c=(c_1,c_2)$ be a constant in $\mathbb{C}^2$. Then any pair of transcendental entire solutions with finite order for the system of Fermat-type difference equations \bea\label{e2.3a} \begin{cases} f_1(z_1,z_2)^2+f_2(z_1+c_1,z_2+c_2)^2=1,\\
f_2(z_1,z_2)^2+f_1(z_1+c_1,z_2+c_2)^2=1
\end{cases}\eea have the following forms \beas (f_1(z),f_2(z))=\left(\frac{e^{L(z)+B_1}+e^{-(L(z)+B_1)}}{2},\frac{A_{21}e^{L(z)+B_1}+A_{22}e^{-(L(z)+B_1)}}{2}\right),\eeas where $L(z)=\alpha_1z_1+\alpha_2z_2$, $B_1$ is a constant in $\mathbb{C}$, and $c, A_{21}, A_{22}$ satisfy one of the following cases \begin{enumerate}
\item [(i)] $L(c)=2k\pi i$, $A_{21}=-i$ and $A_{22}=i$, or $L(c)=(2k+1)\pi i$, $A_{21}=i$ and $A_{22}=-i$, here and below $k$ is an integer;
\item[(ii)] $L(c)=(2k+1/2)\pi i$, $A_{21}=-1$ and $A_{22}=-1$, or $L(c)=(2k-1/2)\pi i$, $A_{21}=1$ and $A_{22}=1$.\end{enumerate}\end{theoC}
\begin{theoD}{\em\cite{Xu-Liu-Li-JMAA-2020}}
Let $c=(c_1,c_2)$ be a constant in $\mathbb{C}^2$. Then any pair of transcendental entire solutions with finite order for the system of Fermat-type partial differential-difference equations\bea\label{e2.4a} \begin{cases} \left(\frac{\partial f_1(z_1,z_2)}{\partial z_1}\right)^2+f_2(z_1+c_1,z_2+c_2)^2=1,\\
\left(\frac{\partial f_2(z_1,z_2)}{\partial z_1}\right)^2+f_1(z_1+c_1,z_2+c_2)^2=1\end{cases}\eea
have the following forms \beas (f_1(z),f_2(z))=\left(\frac{e^{L(z)+B_1}+e^{-(L(z)+B_1)}}{2},\frac{A_{21}e^{L(z)+B_1}+A_{22}e^{-(L(z)+B_1)}}{2}\right),\eeas where $L(z)=\alpha_1z_1+\alpha_2z_2$, $B_1$ is a constant in $\mathbb{C}$, and $c, A_{21}, A_{22}$ satisfy one of the following cases \begin{enumerate}
\item [(i)] $A_{21}=-i$, $A_{22}=i$, and $a_1=i$, $L(c)=\left(2k+\frac{1}{2}\right)\pi i$, or $a_1=-i$, $L(c)=\left(2k-\frac{1}{2}\right)\pi i$;
\item[(ii)] $A_{21}=i$, $A_{22}=-i$, and $a_1=i$, $L(c)=\left(2k-\frac{1}{2}\right)\pi i$, or $a_1=-i$, $L(c)=\left(2k+\frac{1}{2}\right)\pi i$;
\item [(iii)] $A_{21}=1$, $A_{22}=1$, and $a_1=i$, $L(c)=2k\pi i$, or $a_1=-i$, $L(c)=\left(2k+1\right)\pi i$
\item [(iv)] $A_{21}=-1$, $A_{22}=-1$, and $a_1=i$, $L(c)=\left(2k+1\right)\pi i i$, or $a_1=-i$, $L(c)=2k\pi i$.\end{enumerate}
\end{theoD}
\par Theorems A--C suggest us the following open questions.
\begin{ques}\label{q1}
Can we generalize Theorem C from $\mathbb{C}^2$ to $\mathbb{C}^n$, where $n$ is any positive integer?
\end{ques}
\begin{ques}\label{q2}
Can we characterize the entire solutions of the system of equations \eqref{e2.4a} when the right hand sides of each of the equations of \eqref{e2.4a} are replaced by $e^{g_1(z)}$ and $e^{g_2(z)}$, respectively, where $g_1(z)$ and $g_2(z)$ are any two polynomials in $\mathbb{C}^n$?
\end{ques}
\par Inspired by the above two questions and utilizing the Nevanlinna theory and difference Nevanlinna theory for several complex variables (see \cite{Cao & Korhonen & 2016,Korhonen & CMFT & 2012}), we consider the following system of equations \bea \label{e2.1}\begin{cases} f_1(z)^2+f_2(z+c)^2=e^{g_1(z)},\\	f_2(z)^2+f_1(z+c)^2=e^{g_2(z)},\end{cases}\eea
and obtain Theorem \ref{t1} below. Before we state it, let us set the following. \bea\label{e2.5} \Phi(z)&&=\sum_{i_1=1}^{^nC_2}H_{i_1}^2(s_{i_1}^2)+\sum_{i_2=1}^{^nC_3}H_{i_2}^3(s_{i_2}^3)+\sum_{i_3=1}^{^nC_4}H_{i_3}^4(s_{i_3}^4)+\cdots\nonumber\\&&+\sum_{i_{n-2}=1}^{^nC_{n-1}}H_{i_{n-2}}^{n-1}(s_{i_{n-2}}^{n-1})+H_{n-1}^n(s_{n-1}^{n}),\eea
\bea\label{e2.6} \Psi(z)&&=\sum_{i_1=1}^{^nC_2}L_{i_1}^2(t_{i_1}^2)+\sum_{i_2=1}^{^nC_3}L_{i_2}^3(t_{i_2}^3)+\sum_{i_3=1}^{^nC_4}L_{i_3}^4(t_{i_3}^4)+\cdots\nonumber\\&&+\sum_{i_{n-2}=1}^{^nC_{n-1}}L_{i_{n-2}}^{n-1}(t_{i_{n-2}}^{n-1})+L_{n-1}^n(t_{n-1}^{n}),\eea where $H_{i_1}^2$ is a polynomial in $s_{i_1}^2:=d_{i_1j_1}^2z_{j_1}+d_{i_1j_2}^2z_{j_2}$ with $d_{i_1j_1}^2c_{j_1}+d_{i_1j_2}^2c_{j_2}=0$, $1\leq i_1\leq \;^nC_2$, $1\leq j_1<j_2\leq n$, $H_{i_2}^3$ is a polynomial in $s_{i_2}^3:=d_{i_2j_1}^3z_{j_1}+d_{i_2j_2}^3z_{j_2}+d_{i_2j_3}^3z_{j_3}$ with $d_{i_2j_1}^3c_{j_1}+d_{i_2j_2}^3c_{j_2}+d_{i_2j_3}^3z_3=0$, $1\leq i_2\leq \;^nC_3$, $1\leq j_1<j_2<j_3\leq n\ldots$, $H_{i_{n-2}}^{n-1}$ is a polynomial in $s_{i_{n-2}}^{n-1}:=d_{i_{n-2}j_1}z_{j_1}+d_{i_{n-2}j_2}z_{j_2}+\cdots+d_{i_{n-2}j_{n-1}}z_{j_{n-1}}$ with $d_{i_{n-2}j_1}c_{j_1}+d_{i_{n-2}j_2}c_{j_2}+\cdots+d_{i_{n-2}j_{n-1}}c_{j_{n-1}}=0$, $1\leq i_{n-2}\leq \;^nC_{n-1}$, $1\leq j_1<j_2<\cdots<j_{n-1}\leq n$ and $H_{n-1}^{n}$ is a polynomial in $s_{n-1}^n:=d_{i_{n-1}1}z_1+d_{i_{n-1}2}z_2+\cdots+d_{i_{n-1}n}z_n$ with $d_{i_{n-1}1}c_1+d_{i_{n-1}2}c_2+\cdots+d_{i_{n-1}n}c_n=0$, $L_{i_1}^2$ is a polynomial in $t_{i_1}^2:=e_{i_1j_1}^2z_{j_1}+e_{i_1j_2}^2z_{j_2}$ with $e_{i_1j_1}^2c_{j_1}+e_{i_1j_2}^2c_{j_2}=0$, $1\leq i_1\leq \;^nC_2$, $1\leq j_1<j_2\leq n$, $H_{i_2}^3$ is a polynomial in $t_{i_2}^3:=e_{i_2j_1}^3z_{j_1}+e_{i_2j_2}^3z_{j_2}+e_{i_2j_3}^3z_{j_3}$ with $e_{i_2j_1}^3c_{j_1}+e_{i_2j_2}^3c_{j_2}+e_{i_2j_3}^3z_3=0$, $1\leq i_2\leq \;^nC_3$, $1\leq j_1<j_2<j_3\leq n\ldots$, $L_{i_{n-2}}^{n-1}$ is a polynomial in $t_{i_{n-2}}^{n-1}:=e_{i_{n-2}j_1}z_{j_1}+e_{i_{n-2}j_2}z_{j_2}+\cdots+e_{i_{n-2}j_{n-1}}z_{j_{n-1}}$ with $e_{i_{n-2}j_1}c_{j_1}+e_{i_{n-2}j_2}c_{j_2}+\cdots+e_{i_{n-2}j_{n-1}}c_{j_{n-1}}=0$, $1\leq i_{n-2}\leq \;^nC_{n-1}$, $1\leq j_1<j_2<\cdots<j_{n-1}\leq n$ and $L_{n-1}^{n}$ is a polynomial in $t_{n-1}^n:=e_{i_{n-1}1}z_1+e_{i_{n-1}2}z_2+\cdots+e_{i_{n-1}n}z_n$ with $e_{i_{n-1}1}c_1+e_{i_{n-1}2}c_2+\cdots+e_{i_{n-1}n}c_n=0$, where for each $k$, the representation of $s_i^k$ and $t_i^k$ in terms of the conditions of $j_1,j_2,\ldots,j_k$ are unique.\vspace{1.2mm}
\par Now we state our first result as follows.
\begin{theo}\label{t1}
	Let $c=(c_1, c_2,\ldots,c_n)\in \mathbb{C}^n$ and $g_1(z)$, $g_1(z)$ be any two polynomials in $\mathbb{C}^n$. If $(f_1,f_2)$ is a pair of transcendental entire solution with finite order of simultaneous Fermat-type difference equation \eqref{e2.1}, then one of the following cases must occur.
	\begin{enumerate}
		\item[\textbf{(i)}] $g_1(z)=L(z)+\Phi(z)+B_1$, $g_2(z)=L(z)+\Phi(z)+B_2$, where $L(z)=\sum_{j=1}^{n}a_jz_j$, $\Phi(z)$ is a polynomial in $\mathbb{C}^n$ as defined in \eqref{e2.5}, $a_j, B_1,B_2\in\mathbb{C}$ for $j=1,2,\ldots,n$ and 
		\beas f_1(z)=\frac{\xi_1^2+1}{2\xi_1}e^{\frac{1}{2}[L(z)+\Phi(z)+B_1]},\; f_2(z)=\frac{\xi_2^2+1}{2\xi_2}e^{\frac{1}{2}[L(z)+\Phi(z)+B_2]},\eeas where $\xi_j$ is a non-zero complex number in $\mathbb{C}$ such that $\xi_j^4\neq1$, $j=1,2$, and \beas e^{L(c)}=-\frac{(\xi_1^2-1)(\xi_2^2-1)}{(\xi_1^2+1)(\xi_2^2+1)}\;\;\text{and}\;\; e^{B_1-B_2}=\frac{\xi_1^2(\xi_2^4-1)}{\xi_2^2(\xi_1^4-1)}.\eeas
		\item[\textbf{(ii)}] $g_1(z)=L(z)+R(z)+D_1$, $g_2(z)=L(z)+R(z)+D_2$, where $L(z)=L_1(z)+L_2(z)$, $R(z)=\Phi(z)+\Psi(z)$, $D_1=B_1+B_2$, $D_2=B_3+B_4$ with $L_1(z)+\Phi(z)\neq L_2(z)+\Psi(z)$, $L_j(z)=a_{j1}z_1+a_{j2}z_2$, $\Phi(z),\Psi(z)$ are defined in \eqref{e2.6} and \eqref{e2.6}, respectively, $a_{j1}, a_{j2}, D_i\in\mathbb{C}$, $j=1,2,\ldots,n$ $i=1,2$, and \beas f_1(z)=\frac{1}{2}\left[e^{L_1(z)+\Phi(z)+B_1}+e^{L_2(z)+\Psi(z)+B_3}\right],\eeas \beas f_2(z)=\frac{1}{2}\left[e^{L_1(z)+\Phi(z)+B_2}+e^{L_2(z)+\Psi(z),+B_4}\right],\eeas where $L_1(z),L_2(z)$, $B_1,B_2,B_3,B_4$ satisfy one of the following relations.\vspace{1.2mm} \begin{enumerate}
			\item [\textbf{(a)}] $e^{L_1(c)}=i$, $e^{L_2(c)}=i$, $e^{B_1-B_2}=-1$, $e^{B_3-B_4}=1$.\vspace{1.2mm}
			\item [\textbf{(b)}] $e^{L_1(c)}=i$, $e^{L_2(c)}=-i$, $e^{B_1-B_2}=-1$, $e^{B_3-B_4}=-1$.\vspace{1.2mm}
			\item [\textbf{(c)}] $e^{L_1(c)}=-i$, $e^{L_2(c)}=i$, $e^{B_1-B_2}=1$, $e^{B_3-B_4}=1$.\vspace{1.2mm}
			\item [\textbf{(d)}] $e^{L_1(c)}=-i$, $e^{L_2(c)}=-i$, $e^{B_1-B_2}=1$, $e^{B_3-B_4}=-1$.\vspace{1.2mm}
		\end{enumerate}
		\item[\textbf{(iii)}] $g_1(z)=L(z)+R(z)+D_1$, $g_2(z)=L(z)+R(z)+D_2$, where $L(z)=L_1(z)+L_2(z)$, $R(z)=\Psi(z)+\chi(z)$, $D_1=B_1+B_3$, $D_2=B_2+B_4$ with $L_1(z)+\Phi(z)\neq L_2(z)+\Psi(z)$, $L_j(z)=a_{j1}z_1+a_{j2}z_2$, $\Phi(z),\Psi(z)$ are defined in \eqref{e2.6} and \eqref{e2.6},respectively, $a_{j1}, a_{j2}, D_i\in\mathbb{C}$, $j=1,2,\ldots,n$ $i=1,2$, and \beas f_1(z)=\frac{1}{2}\left[e^{L_1(z)+\Phi(z)+B_1}+e^{L_2(z)+\chi(z)+B_3}\right],\eeas \beas f_2(z)=\frac{1}{2}\left[e^{L_1(z)+\Phi(z)+B_2}+e^{L_2(z)+\chi(z),+B_4}\right],\eeas where $L_1(z),L_2(z)$, $B_1,B_2,B_3,B_4$ satisfy one of the following relations.\vspace{1.2mm} \begin{enumerate}
			\item [\textbf{(a)}] $e^{L_1(c)}=1$, $e^{L_2(c)}=1$, $e^{B_1-B_2}=-i$, $e^{B_3-B_4}=i$.\vspace{1.2mm}
			\item [\textbf{(b)}] $e^{L_1(c)}=1$, $e^{L_2(c)}=-1$, $e^{B_1-B_2}=-i$, $e^{B_3-B_4}=-i$.\vspace{1.2mm}
			\item [\textbf{(c)}] $e^{L_1(c)}=-1$, $e^{L_2(c)}=1$, $e^{B_1-B_2}=i$, $e^{B_3-B_4}=i$.\vspace{1.2mm}
			\item [\textbf{(d)}] $e^{L_1(c)}=-1$, $e^{L_2(c)}=-1$, $e^{B_1-B_2}=i$, $e^{B_3-B_4}=-i$.\end{enumerate}\end{enumerate}
\end{theo}
\par The following examples show the existence of transcendental entire solutions with finite order of the system \eqref{e2.1}.
\begin{exm}
	Let $\xi_1=\xi_2=2$, $L(z)=z_1+2z_2$. Choose $c=(c_1,c_2)\in\mathbb{C}^2$ such that $e^{c_1+2c_2}=-9/25$ and $\Phi(z)=(c_2z_1-c_1z_2)^5$. Let $g_1(z_1,z_2)=L(z)+\Phi(z)+B_1$ and $g_2(z_1,z_2)=L(z)+\Phi(z)+B_2$, where $B_1,B_2$ are constants in $\mathbb{C}$ such that $B_1-B_2=2\pi i$. Then in view of the conclusion \textbf{(i)} of Theorem $\ref{t1}$, we can see that \beas (f_1,f_2)=\left(\frac{5}{4}e^{\frac{1}{2}[z_1+2z_2+(c_2z_1-c_1z_2)^3+B_1]},\frac{5}{4}e^{\frac{1}{2}[z_1+2z_2+(c_2z_1-c_1z_2)^3+B_2]}\right)\eeas is a solution of \eqref{e2.1}.
\end{exm}
\begin{exm}
	Let $c_1=7\pi i/5, c_2=-2\pi i/5$, $L_1(z)=z_1+z_2$, $L_2(z)=2z_1-3z_2$, $\Phi(z)=-\dfrac{\pi^2}{25}(2z_1+7z_2)^2$, $\Psi(z)\equiv0$, $B_1=3\pi i/2$, $B_2=\pi i$, $B_3=2\pi i$, $B_4=\pi i/2$ and 
	\beas f_1(z_1,z_2)=\frac{1}{2}\left[e^{z_1+z_2-\frac{\pi^2}{25}(2z_1+7z_2)^2+\frac{3\pi i}{2}}+e^{2z_1-3z_2+2\pi i}\right],\eeas
	\beas f_2(z_1,z_2)=\frac{1}{2}\left[e^{z_1+z_2-\frac{\pi^2}{25}(2z_1+7z_2)^2+\pi i}+e^{2z_1-3z_2+\frac{\pi i}{2}}\right].\eeas Then in view of the conclusion \textbf{(ii)} of Theorem $\ref{t1}$, we can see that $(f_1,f_2)$ is a solution of \eqref{e2.1} with
	$g_1(z_1,z_2)=3z_1-2z_2-\dfrac{\pi^2}{25}(2z_1+7z_2)^2+\dfrac{7\pi i}{2}$ and $g_2(z_1,z_2)=3z_1-2z_2-\dfrac{\pi^2}{25}(2z_1+7z_2)^2+\dfrac{3\pi i}{2}$
\end{exm}
\par In order to generalize Theorem D to a large extent, let us consider the following system of equations.
\bea\label{e2.3} \begin{cases} \left(\frac{\partial^k f_1(z_1,z_2)}{\partial z_1^k}\right)^2+f_2(z_1+c_1,z_2+c_2)^2=e^{g_1(z_1,z_2)},\vspace{1mm}\\	\left(\frac{\partial^k f_2(z_1,z_2)}{\partial z_1^k}\right)^2+f_1(z_1+c_1,z_2+c_2)^2=e^{g_2(z_1,z_2)},\end{cases}\eea where $e^{g_1(z_1,z_2)}$, $e^{g_2(z_1,z_2)}$ are any two polynomials in $\mathbb{C}^2$ and obtain the following result.
\begin{theo}\label{t2}
Let $c=(c_1, c_2)\in \mathbb{C}^2$ and $g_1(z_1,z_2)$, $g_1(z_1,z_2)$ be any two polynomials in $\mathbb{C}^2$. If $(f_1,f_2)$ is a pair of transcendental entire solution with finite order of simultaneous Fermat type difference equation \eqref{e2.3}, then one of the following cases must occur.
\begin{enumerate}
\item[\textbf{(i)}] $g_1(z_1,z_2)=L(z)+H(s)+B_1$, $g_2(z_1,z_2)=L(z)+H(s)+B_2$, where $L(z)=a_1z_1+a_2z_2$, $H(s)$ is a polynomial in $s:=d_2z_2$ with $d_2c_2=0$, $a_1,a_2,B_1,B_2$ constants in $\mathbb{C}$, and \beas f_1(z_1,z_2)=\frac{\xi_2^2-1}{2i\xi_2}e^{\frac{1}{2}[L(z)+H(s)+B_2-L(c)]},\; f_2(z)=\frac{\xi_1^2-1}{2i\xi_1}e^{\frac{1}{2}[L(z)+H(s)+B_1-L(c)]},\eeas where $\xi_j$ is a non-zero complex number in $\mathbb{C}$ such that $\xi_j^4\neq0,1$, $j=1,2$, and \beas e^{L(c)}=-\left(\frac{a_1}{2}\right)^k\frac{(\xi_1^2-1)(\xi_2^2-1)}{(\xi_1^2+1)(\xi_2^2+1)}\;\;\text{and}\;\; e^{B_1-B_2}=\frac{\xi_1^2(\xi_2^4-1)}{\xi_2^2(\xi_1^4-1)}.\eeas \vspace{1.2mm}
\item[\textbf{(ii)}] $g_1(z_1,z_2)=L_1(z)+L_2(z)+H_1(s)+H_2(s)+B_1+B_3$, $g_2(z_1,z_2)=L_1(z)+L_2(z)+H_1(s)+H_2(s)+B_2+B_4$, where $L_1(z)+H_1(s)\neq L_2(z)+H_2(s)$, $L_j(z)=a_{j1}z_1+a_{j2}z_2$, $H_j(s)$ is a polynomial in $s:=d_2z_2$ with $d_2c_2=0$, $a_{11},a_{12}, a_{21},a_{22}, B_1,B_2,B_3,B_4\in\mathbb{C}$, and \beas f_1(z_1,z_2)=\frac{1}{2i}\left[e^{L_1(z)+H_1(s)-L_1(c)+B_2}-e^{L_2(z)+H_2(s)-L_2(c)+B_4}\right],\eeas \beas f_2(z_1,z_2)=\frac{1}{2i}\left[e^{L_1(z)+H_1(s)-L_1(c)+B_1}-e^{L_2(z)+H_2(s)-L_2(c)+B_3}\right],\eeas where $L_1(z),L_2(z)$, $B_1,B_2,B_3,B_4$ satisfies one of the following relations\vspace{1.2mm} \begin{enumerate}
\item [\textbf{(a)}] $e^{L_1(c)}=ia_{11}^k$, $e^{L_2(c)}=ia_{21}^k$, $e^{B_1-B_2}=-1$, $e^{B_3-B_4}=1$.\vspace{1.2mm}
\item [\textbf{(b)}] $e^{L_1(c)}=ia_{11}^k$, $e^{L_2(c)}=-ia_{21}^k$, $e^{B_1-B_2}=-1$, $e^{B_3-B_4}=-1$.\vspace{1.2mm}
\item [\textbf{(c)}] $e^{L_1(c)}=-ia_{11}^k$, $e^{L_2(c)}=ia_{21}^k$, $e^{B_1-B_2}=1$, $e^{B_3-B_4}=1$.\vspace{1.2mm}
\item [\textbf{(d)}] $e^{L_1(c)}=-ia_{11}^k$, $e^{L_2(c)}=-ia_{21}^k$, $e^{B_1-B_2}=1$, $e^{B_3-B_4}=-1$.\vspace{1.2mm}
\end{enumerate}
\item[\textbf{(iii)}] $g_1(z_1,z_2)=L_1(z)+L_2(z)+H_1(s)+H_2(s)+B_1+B_3$, $g_2(z_1,z_2)=L_1(z)+L_2(z)+H_1(s)+H_2(s)+B_2+B_4$, where $L_1(z)+H_1(s)\neq L_2(z)+H_2(s)$, $L_j(z)=a_{j1}z_1+a_{j2}z_2$, $H_j(s)$ is a polynomial in $s:=d_2z_2$ with $d_2c_2=0$, $a_{11},a_{12}, a_{21},a_{22}, B_1,B_2,B_3,B_4\in\mathbb{C}$, and \beas f_1(z_1,z_2)=-\frac{1}{2i}\left[e^{L_1(z)+H_1(s)-L_1(c)+B_2}-e^{L_2(z)+H_2(s)-L_2(c)+B_4}\right],\eeas \beas f_2(z_1,z_2)=\frac{1}{2i}\left[e^{L_1(z)+H_1(s)-L_1(c)+B_1}-e^{L_2(z)+H_2(s)-L_2(c)+B_3}\right],\eeas where $L_1(z),L_2(z)$, $B_1,B_2,B_3,B_4$ satisfies one of the following relations\vspace{1.2mm} \begin{enumerate}
\item [\textbf{(a)}] $e^{L_1(c)}=a_{11}^k$, $e^{L_2(c)}=a_{21}^k$, $e^{B_1-B_2}=i$, $e^{B_3-B_4}=-i$.\vspace{1.2mm}
\item [\textbf{(b)}] $e^{L_1(c)}=a_{11}^k$, $e^{L_2(c)}=-a_{21}^k$, $e^{B_1-B_2}=i$, $e^{B_3-B_4}=i$.\vspace{1.2mm}
\item [\textbf{(c)}] $e^{L_1(c)}=-a_{11}^k$, $e^{L_2(c)}=a_{21}^k$, $e^{B_1-B_2}=-i$, $e^{B_3-B_4}=-i$.\vspace{1.2mm}
\item [\textbf{(d)}] $e^{L_1(c)}=-a_{11}^k$, $e^{L_2(c)}=-a_{21}^k$, $e^{B_1-B_2}=-i$, $e^{B_3-B_4}=i$.\vspace{1.2mm}
\end{enumerate}\end{enumerate}\end{theo}
\par The following examples show the existence of transcendental entire solutions with finite order of the system \eqref{e2.3}.\vspace{1.2mm}
\begin{exm}
Let $\xi_1=\xi_2=2$, $L(z)=2z_1+z_2$, $g_1(z_1,z_2)=g_2(z_1,z_2)=L(z)+\pi i$. Choose $c=(c_1,c_2)\in\mathbb{C}^2$ such that $e^{L(c)}=-9/25$. Then in view of the conclusion $(i)$ of Theorem {\em\ref{t2}}, it follows that $(f_1(z),f_2(z))$ is a solution of \beas \begin{cases} \left(\frac{\partial f_1(z_1,z_2)}{\partial z_1}\right)^2+f_2(z_1+c_1,z_2+c_2)^2=e^{g_1(z_1,z_2)},\vspace{1mm}\\	\left(\frac{\partial f_2(z_1,z_2)}{\partial z_1}\right)^2+f_1(z_1+c_1,z_2+c_2)^2=e^{g_2(z_1,z_2)},\end{cases},\eeas where \beas f_1(z)=f_2(z)= -\frac{5i}{4}e^{\frac{1}{2}[2z_1+z_2]}.\eeas	
\end{exm}
\begin{exm}
Let $k=2$, $L_1(z)=z_1+z_2$, $L_2(z)=2z_1+z_2$, $(c_1,c_2)=(2\log2,\pi i/2-2\log2)$, $B_1=2\pi i$, $B_2=\pi i$, $B_3=B_4=\pi i$. Then in view of the conclusion $(ii)(a)$ of Theorem {\em\ref{t2}}, we can easily see that \beas(f_1(z),f_2(z))=\left(\frac{1}{8}\left(4e^{z_1+z_2}-e^{2z_1+z_2}\right),-\frac{1}{8}\left(4e^{z_1+z_2}+e^{2z_1+z_2}\right)\right)\eeas is a solution of \beas \begin{cases} \left(\frac{\partial^2 f_1(z_1,z_2)}{\partial z_1^2}\right)^2+f_2(z_1+c_1,z_2+c_2)^2=e^{g_1(z_1,z_2)},\vspace{1mm}\\	\left(\frac{\partial^2 f_2(z_1,z_2)}{\partial z_1^2}\right)^2+f_1(z_1+c_1,z_2+c_2)^2=e^{g_2(z_1,z_2)},\end{cases},\eeas where $g_1(z_1,z_2)=3z_1+2z_2+3\pi i$ and $g_1(z_1,z_2)=3z_1+2z_2+2\pi i$	
\end{exm}
\begin{exm}
Let $k=1$, $L_1(z)=2z_1+3z_2$, $L_2(z)=2z_1-z_2$, $(c_1,c_2)=(\frac{1}{2}\log2,0)$, $B_1=\pi i$, $B_2=\pi i/2$, $B_3=\pi i/2$, $B_4=\pi i$. Then in view of the conclusion $(iii)(a)$ of Theorem {\em\ref{t2}}, we can easily see that \beas(f_1(z),f_2(z))=\left(-\frac{1}{4i}\left(ie^{2z_1+3z_2+z_2^3}+e^{2z_1-z_2+z_2^5}\right),-\frac{1}{4i}\left(e^{2z_1+3z_2+z_2^3}+ie^{2z_1-z_2+z_2^5}\right)\right)\eeas is a solution of \beas \begin{cases} \left(\frac{\partial f_1(z_1,z_2)}{\partial z_1}\right)^2+f_2(z_1+c_1,z_2+c_2)^2=e^{g_1(z_1,z_2)},\vspace{1mm}\\	\left(\frac{\partial f_2(z_1,z_2)}{\partial z_1}\right)^2+f_1(z_1+c_1,z_2+c_2)^2=e^{g_2(z_1,z_2)},\end{cases},\eeas where $g_1(z_1,z_2)=g_1(z_1,z_2)=4z_1+2z_2+z_2^3+z_2^5+3\pi i/2$.\end{exm}
\par Beside these, we are interested to characterize the entire solutions of the following system of $k$-th partial differential difference equations 
\bea\label{e2.4} \begin{cases} \left(\frac{\partial^k f_1(z_1,z_2)}{\partial z_1^k}\right)^2+(f_2(z_1+c_1,z_2+c_2)-f_2(z_1,z_2))^2=1,\vspace{1mm}\\\left(\frac{\partial^k f_2(z_1,z_2)}{\partial z_1^k}\right)^2+(f_1(z_1+c_1,z_2+c_2)-f_1(z_1,z_2))^2=1,\end{cases}\eea which has not been discussed earlier in the literature. Regarding this, we obtain the following result.
\begin{theo}\label{t3}
Let $c$ be a non-zero constant in $\mathbb{C}^2$, $k$ be a positive integer and $f_1(z), f_2(z)$ are transcendental entire functions of finite order not of $c$-periodic. If $(f_1(z),f_2(z))$ be a solution with finite order of \eqref{e2.4}, then one of the following conclusions must hold.\vspace{1.2mm}
\begin{enumerate}\item [\textbf{(i)}] \beas f_1(z_1,z_2)=\frac{e^{L(z)+H(z_2)+\beta}-e^{-(L(z)+H(z_2)+\beta)}}{2\alpha_1^k},\eeas\beas f_2(z_1,z_2)=\frac{e^{L(z)+H(z_2)+\beta+\eta}-e^{-(L(z)+H(z_2)+\beta+\eta)}}{2\alpha_1^k}.\eeas  where $L(z)=\alpha_1z_1+\alpha_2z_2$, $H(z_2)$ is a polynomial in $z_2$, only, $k$ is an odd positive integer, $\alpha_1(\neq0),\alpha_2, \beta,\eta\in\mathbb{C}$ and satisfy the following relations:\vspace{1.2mm} \beas e^{L(c)}=\frac{ie^{\eta}}{\alpha_1^k+ie^{\eta}},\;\; e^{2\eta}=1\;\;\text{and}\;\; \alpha_1^{k}=-2ie^{-\eta}.\eeas 
\item [\textbf{(ii)}] \beas f_1(z_1,z_2)=\frac{e^{L(z)+H(z_2)+\beta}+e^{-(L(z)+H(z_2)+\beta)}}{2\alpha_1^k},\eeas\beas f_2(z_1,z_2)=\frac{e^{-(L(z)+H(z_2)+\beta)+\eta}+e^{L(z)+H(z_2)+\beta-\eta}}{2\alpha_1^k}.\eeas where $L(z)=\alpha_1z_1+\alpha_2z_2$, $H(z_2)$ is a polynomial in $z_2$, only $k$ is an even positive integer, $\alpha_1(\neq0),\alpha_2, \beta,\eta\in\mathbb{C}$ and satisfy the following relations:\vspace{1.2mm}\beas e^{L(c)}=1+ie^{-\eta}\alpha_1^k,\;\; e^{2\eta}=-1\;\;\text{and}\;\; \alpha_1^{k}=2ie^{\eta}.\eeas\end{enumerate}\end{theo}
\par The following examples show that transcendental entire solutions of \eqref{e2.4} are precise.\vspace{1.2mm}
\begin{exm}Let $k=1$, $\eta=4\pi i$ and $\alpha_1=-2i$. Choose $c=(c_1,c_2)\in\mathbb{C}^2$ such that $-2ic_1+c_2=(2m+\pi i)$, $m$ being an integer. Then, it can be easily verified that \beas (f_1(z),f_2(z))=\left(\frac{e^{2iz_1-z_2-z_2^{10}}-e^{-2iz_1+z_2+z_2^{10}}}{4i},\frac{e^{2iz_1-z_2-z_2^{10}}-e^{-2iz_1+z_2+z_2^{10}}}{4i}\right)\eeas is a solution of \beas \begin{cases}\left(\frac{\partial f_1}{\partial z_1}\right)^2+[f_2(z+c)-f_2(z)]^2=1,\\\left(\frac{\partial f_2}{\partial z_1}\right)^2+[f_1(z+c)-f_1(z)]^2=1.\end{cases}\eeas\end{exm}
\begin{exm}	Let $k=2$, $\eta=5\pi i/2$. Choose $\alpha_1\in\mathbb{C}$ such that $\alpha_1^2=-2$ and $c=(c_1,c_2)\in\mathbb{C}^2$ such that $\alpha_1c_1+c_2+1=(2m+1)\pi i$, $m\in\mathbb{Z}$. Then one can see that \beas (f_1(z),f_2(z))=\left(\frac{e^{\alpha_1z_1+z_2+1}+e^{(-\alpha_1z_1+z_2+1)}}{-4},\frac{e^{\alpha_1z_1+z_2+1}-e^{-(\alpha_1z_1+z_2+1)}}{-4i}\right)\eeas is a solution of \beas\begin{cases}\left(\frac{\partial^2 f_1}{\partial z_1^2}\right)^2+[f_2(z+c)-f_2(z)]^2=1,\vspace{.5mm}\\\left(\frac{\partial^2 f_2}{\partial z_1^2}\right)^2+[f_1(z+c)-f_1(z)]^2=1.\end{cases}\eeas\end{exm}
\section{Proof of the main results}
\par Before we start the proof of the main results, we present some important lemmas which will play key role to prove the main results of the paper.
\begin{lem}\label{lem3.1}{\em\cite{Hu & Li & Yang & 2003}}	Let $f_j\not\equiv0$ $(j=1,2\ldots,m;\; m\geq 3)$ be meromorphic functions on $\mathbb{C}^{n}$ such that $f_1,\ldots, f_{m-1}$ are not constants, $f_1+f_2+\cdots+ f_m=1$ and such that\beas \sum_{j=1}^{m}\left\{N_{n-1}\left(r,\frac{1}{f_j}\right)+(m-1)\ol N(r,f_j)\right\}< \lambda T(r,f_j)+O(\log^{+}T(r,f_j))\eeas holds for $j=1,\ldots, m-1$ and all $r$ outside possibly a set with finite logarithmic measure, where $\lambda < 1$ is a positive number. Then $f_m = 1$.	\end{lem}
\begin{lem}\label{lem3.2}{\em\cite{Lelong & 1968, ronkin & AMS & 1971, Stoll & AMS & 1974}}
	For an entire function $F$ on $\mathbb{C}^n$, $F(0)\not\equiv 0$ and put $\rho(n_F)=\rho<\infty$. Then there exist a canonical function $f_F$ and a function $g_F\in\mathbb{C}^n$ such that $F(z)=f_F (z)e^{g_F(z)}$. For the special case $n=1$, $f_F$ is the canonical product of Weierstrass.
\end{lem}
\begin{lem}\label{lem3.3}{\em\cite{P`olya & Lond & 1926}}
	If $g$ and $h$ are entire functions on the complex plane $\mathbb{C}$ and $g(h)$ is an entire function of finite order, then there are only two possible	cases: either
	\begin{enumerate}
		\item [(i)] the internal function $h$ is a polynomial and the external function $g$ is of finite order; or else
		\item [(ii)] the internal function $h$ is not a polynomial but a function of finite order, and the external function $g$ is of zero order.
\end{enumerate}\end{lem}

\begin{lem}\label{lem3.1a}{\em\cite{Hu & Li & Yang & 2003}}
	Let $f_j\not\equiv0$ $(j=1,2,3)$ be meromorphic functions on $\mathbb{C}^{n}$ such that $f_1$ are not constant, $f_1+f_2+f_3=1$, and such that \beas \sum_{j=1}^{3}\left\{N_{2}\left(r,\frac{1}{f_j}\right)+2\ol N(r,f_j)\right\}< \lambda T(r,f_j)+O(\log^{+}T(r,f_j))\eeas
	holds for all $r$ outside possibly a set with finite logarithmic measure, where $\lambda < 1$ is a positive number. Then, either $f_2 = 1$ or $f_3=1$.	
\end{lem}
\begin{lem}\label{lem3.7}{\em\cite{Hu & Li & Yang & 2003}}
	Let $a_0(z), a_1(z),\ldots, a_n(z)$ $(n\geq1)$ be meromorphic functions on $\mathbb{C}^m$ and $g_0(z), g_1(z),\ldots,g_n(z)$ are entire functions on $\mathbb{C}^m$ such that $g_j(z)-g_k(z)$ are not constants for $0\leq j<k\leq n$. If $\sum_{j=0}^{n}a_j(z)e^{g_j(z)}\equiv0$, and $|| T(r,a_j)=o(T(r))$, where $T(r)=\text{min}_{0\leq j<k\leq n} T(r,e^{g_j-g_k})$ for $j=0,1,\ldots,n$, then $a_j(z)\equiv0$ for each $j=0,1,\ldots,n$.
\end{lem}


\begin{proof}[\textbf{Proof of Theorem $\ref{t1}$}]
Suppose that $(f_1,f_2)$ is a pair of transcendental entire functions of finite order satisfying system \eqref{e2.1}. First we write \eqref{e2.1} as follows.	
\bea\label{e3.1} \begin{cases}		\left[\frac{f_1(z)}{e^{\frac{g_1(z)}{2}}}+i\frac{f_2(z+c)}{e^{\frac{g_1(z)}{2}}}\right]\left[\frac{f_1(z)}{e^{\frac{g_1(z)}{2}}}-i\frac{f_2(z+c)}{e^{\frac{g_1(z)}{2}}}\right]=1,\vspace{1mm}\\\left[\frac{f_2(z)}{e^{\frac{g_2(z)}{2}}}+i\frac{f_1(z+c)}{e^{\frac{g_2(z)}{2}}}\right]\left[\frac{f_2(z)}{e^{\frac{g_2(z)}{2}}}-i\frac{f_1(z+c)}{e^{\frac{g_2(z)}{2}}}\right]=1.	\end{cases}\eea
\par Since $f_1$, $f_2$ are transcendental entire functions with finite order, in view of Lemmas \ref{lem3.2} and \ref{lem3.3}, there exist polynomials $p_1(z)$, $p_2(z)$ in $\mathbb{C}^n$ such that	\bea\label{e3.2}\begin{cases}		\frac{f_1(z)}{e^{\frac{g_1(z)}{2}}}+i\frac{f_2(z+c)}{e^{\frac{g_1(z)}{2}}}=e^{p_1(z)},\vspace{1mm}\\\frac{f_1(z)}{e^{\frac{g_1(z)}{2}}}-i\frac{f_2(z+c)}{e^{\frac{g_1(z)}{2}}}=e^{-p_1(z)},\vspace{1mm}\\\frac{f_2(z)}{e^{\frac{g_2(z)}{2}}}+i\frac{f_1(z+c)}{e^{\frac{g_2(z)}{2}}}=e^{p_2(z)},\vspace{1mm}\\\frac{f_2(z)}{e^{\frac{g_2(z)}{2}}}-i\frac{f_1(z+c)}{e^{\frac{g_2(z)}{2}}}=e^{-p_2(z)}.\end{cases}\eea
	\par Set \bea\label{e3.3}\begin{cases}
		\alpha_1=\frac{1}{2}g_1(z)+p_1(z),\;\; \alpha_2=\frac{1}{2}g_1(z)-p_1(z),\\\alpha_3=\frac{1}{2}g_2(z)+p_2(z),\;\; \alpha_4=\frac{1}{2}g_2(z)-p_2(z).
	\end{cases}\eea
	\par In view of \eqref{e3.2} and \eqref{e3.3}, it follows that \bea\label{e3.4} \begin{cases}
		f_1(z)=\displaystyle\frac{1}{2}\left[e^{\alpha_1(z)}+e^{\alpha_2(z)}\right]\\ f_2(z+c)=\displaystyle\frac{1}{2i}\left[e^{\alpha_1(z)}-e^{\alpha_2(z)}\right]\\f_2(z)=\displaystyle\frac{1}{2}\left[e^{\alpha_3(z)}+e^{\alpha_4(z)}\right]\\f_1(z+c)=\displaystyle\frac{1}{2i}\left[e^{\alpha_3(z)}-e^{\alpha_4(z)}\right].\end{cases}\eea
	\par After simple calculation, it follows from \eqref{e3.4} that \bea\label{e3.5}\begin{cases} ie^{\alpha_1(z+c)-\alpha_3(z)}+ie^{\alpha_2(z+c)-\alpha_3(z)}+e^{\alpha_4(z)-\alpha_3(z)}=1, \\ ie^{\alpha_3(z+c)-\alpha_1(z)}+ie^{\alpha_4(z+c)-\alpha_1(z)}+e^{\alpha_2(z)-\alpha_1(z)}=1.\end{cases}\eea
	
	\par Now, we consider the following four possible cases.\vspace{1mm}
	\par \textbf{Case 1:} Let $\alpha_4(z)-\alpha_3(z)=\eta_1$ and $\alpha_2(z)-\alpha_1(z)=\eta_2$, where $\eta_1,\eta_2$ are constants in $\mathbb{C}$.\vspace{1.2mm}
\par Then, it clearly follows from \eqref{e3.4} that $p_1(z)$ and $p_2(z)$ both are constants in $\mathbb{C}$. Let $e^{p_1}=\xi_1$ and $e^{p_2}=\xi_2$. Then, \eqref{e3.4} can be rewritten as \bea\label{e3.6} \begin{cases}		f_1(z)=\frac{\xi_1+\xi_1^{-1}}{2}e^{\frac{1}{2}g_1(z)},\\ f_2(z+c)=\frac{\xi_1-\xi_1^{-1}}{2i}e^{\frac{1}{2}g_1(z)},\\f_2(z)=\frac{\xi_2+\xi_2^{-1}}{2}e^{\frac{1}{2}g_2(z)},\\f_1(z+c)=\frac{\xi_2-\xi_2^{-1}}{2i}e^{\frac{1}{2}g_2(z)}.\end{cases}\eea
\par After simple computations, it follows from \eqref{e3.6} that \bea\label{e3.7} \begin{cases} i(\xi_1+\xi_1^{-1})e^{\frac{1}{2}[g_1(z+c)-g_2(z)]}=\xi_2-\xi_2^{-1}\\i(\xi_2+\xi_2^{-1})e^{\frac{1}{2}[g_2(z+c)-g_1(z)]}=\xi_1-\xi_1^{-1}.\end{cases}\eea 
\par Since $\xi_j^4\neq0,1$ and $g_j(z_1,z_2)$ are polynomials in $\mathbb{C}^n$ for $=1,2$, it follows from \eqref{e3.7} that $g_1(z+c)-g_2(z)=\zeta_1$ and $g_2(z+c)-g_1(z)=\zeta_2$, where $\zeta_1,\zeta_2\in\mathbb{C}$. This implies that $g_1(z+2c)-g_2(z)=g_2(z+2c)-g_1(z)=\zeta_1+\zeta_2$. Thus, $g_1(z)=L(z)+\Phi(z)+B_1$ and $g_2(z)=L(z)+\Phi(z)+B_2$, where $L(z)=\sum_{j=1}^{n}a_jz_j$, $\Phi(z)$ is a polynomial in $\mathbb{C}^n$ defined in \eqref{e2.5}. Thus, from \eqref{e3.7}, we get 
\bea\label{e3.8} \begin{cases}		e^{\frac{1}{2}[L(c)+B_1-B_2]}=\frac{\xi_2-\xi_2^{-1}}{i(\xi_1+\xi_1^{-1})}\\e^{\frac{1}{2}[L(c)+B_2-B_1]}=\frac{\xi_1-\xi_1^{-1}}{i(\xi_2+\xi_2^{-1})}.\end{cases}\eea
\par Therefore, from \eqref{e3.8}, we obtain \beas e^{L(c)}=-\frac{(\xi_1^2-1)(\xi_2^2-1)}{(\xi_1^2+1)(\xi_2^2+1)}\;\;\text{and}\;\; e^{B_1-B_2}=\frac{\xi_1^2(\xi_2^4-1)}{\xi_2^2(\xi_1^2+1)}.\eeas
\par \textbf{Case 2.} Let $\alpha_4(z)-\alpha_3(z)$ and $\alpha_2(z)-\alpha_1(z)$ both are non-constants. Then, by Lemma \ref{lem3.1a}, it follows from \eqref{e3.5} that \beas ie^{\alpha_1(z+c)-\alpha_3(z)}=1,\;\;\text{or}\;\;ie^{\alpha_2(z+c)-\alpha_3(z)}=1\eeas and \beas ie^{\alpha_3(z+c)-\alpha_1(z)}=1,\;\;\text{or}\;\;ie^{\alpha_4(z+c)-\alpha_1(z)}=1.\eeas
\par Now, we consider the following subcases.\vspace{1.2mm}
\par \textbf{Subcase 2.1.} Let \bea\label{e3.9}		ie^{\alpha_1(z+c)-\alpha_3(z)}=1\;\;\text{and}\;\;ie^{\alpha_3(z+c)-\alpha_1(z)}=1.\eea
\par Then, from \eqref{e3.5} and \eqref{e3.9}, we obtain \bea\label{e3.10} -ie^{\alpha_2(z+c)-\alpha_4(z)}=1\;\;\text{and}\;\;-ie^{\alpha_4(z+c)-\alpha_1(z)}=1.\eea
\par Since $\alpha_j$'s, $j=1,2,3,4$ are all polynomials in $\mathbb{C}^n$, it follows from \eqref{e3.9} and \eqref{e3.10} that $\alpha_1(z+c)-\alpha_3(z)=\eta_1$, $\alpha_3(z+c)-\alpha_1(z)=\eta_2$, $\alpha_2(z+c)-\alpha_4(z)=\eta_3$ and $\alpha_4(z+c)-\alpha_2(z)=\eta_4$, $\eta_j\in\mathbb{C}$ for $j=1,2,3,4$. This implies that $\alpha_1(z+2c)-\alpha_1(z)=\eta_1+\eta_2=\alpha_3(z+2c)-\alpha_3(z)$ and $\alpha_2(z+2c)-\alpha_2(z)=\eta_3+\eta_4=\alpha_4(z+2c)-\alpha_4(z)$. Thus, $\alpha_1(z)=L_1(z)+\Phi(z)+B_1$, $\alpha_3(z)=L_1(z)+\Phi(z)+B_2$, $\alpha_2(z)=L_2(z)+\Psi(z)+B_3$ and $\alpha_4(z)=L_2(z)+\Psi(z)+B_4$, where $L_i(z)=\sum_{j=1}^{n}a_{ij}z_j$, $i=1,2$ and $\Phi(z),\Psi(z)$ are polynomials in $\mathbb{C}^n$ defined as in \eqref{e2.5} and \eqref{e2.6}, respectively, $a_{1j},a_{2j}, B_i\in\mathbb{C}$ for $j=1,2,\ldots,n$ and $i=1,\ldots,4$. Thus, $g_1(z)=L(z)+R(z)+D_1$ and $g_2(z)=L(z)+R(z)+D_2$, where $L(z)=L_1(z)+L_2(z)$, $R(z)=\Phi(z)+\Psi(z)$, $D_1=B_1+B_2$ and $D_2=B_3+B_4$.\vspace{1.2mm}
\par As $\alpha_4-\alpha_3$ and $\alpha_2-\alpha_1$ both are non-constants, it follows that $L_1(z)+\Phi(z)\neq L_2(z)+\Psi(z)$. Therefore, from \eqref{e3.9} and \eqref{e3.10}, we obtain that\bea\label{e3.11} \begin{cases}ie^{L_1(c)+B_1-B_2}=1,\;\;ie^{L_1(c)+B_2-B_1}=1,\\ie^{L_2(c)+B_3-B_4}=-1,\;\;ie^{L_2(c)+B_4-B_2}=-1.\end{cases}\eea
\par After simple calculation, it follows from \eqref{e3.11} that \beas e^{2L_1(c)}=-1,\;\;e^{2L_2(c)}=-1,\;\;e^{2(B_1-B_2)}=1\;\;\text{and}\;\; e^{2(B_3-B_4)}=1.\eeas 
\par From the above relations, we have $e^{L_1(c)}=\pm i$, $e^{L_2(c)}=\pm i$, $e^{B_1-B_2}=\pm 1$ and $e^{B_3-B_4}=\pm 1$. Therefore, in view of \eqref{e3.11}, we obtain the following four possible  relations.\vspace{1.2mm}
\begin{enumerate}
\item [\textbf{(a)}] $e^{L_1(c)}=i$, $e^{L_2(c)}=i$, $e^{B_1-B_2}=-1$, $e^{B_3-B_4}=1$.\vspace{1.2mm}
\item [\textbf{(b)}] $e^{L_1(c)}=i$, $e^{L_2(c)}=-i$, $e^{B_1-B_2}=-1$, $e^{B_3-B_4}=-1$.\vspace{1.2mm}
\item [\textbf{(c)}] $e^{L_1(c)}=-i$, $e^{L_2(c)}=i$, $e^{B_1-B_2}=1$, $e^{B_3-B_4}=1$.\vspace{1.2mm}
\item [\textbf{(d)}] $e^{L_1(c)}=-i$, $e^{L_2(c)}=-i$, $e^{B_1-B_2}=1$, $e^{B_3-B_4}=-1$.\vspace{1.2mm}
\end{enumerate}
\par Hence, it follows from \eqref{e3.4} that \beas\begin{cases} f_1(z)=\frac{1}{2}\left[e^{L_1(z)+\Phi(z)+B_1}+e^{L_2(z)+\Psi(z)+B_3}\right],\vspace{1mm}\\ f_2(z)=\frac{1}{2}\left[e^{L_1(z)+\Phi(z)+B_2}+e^{L_2(z)+\Psi(z)+B_4}\right].\end{cases}\eeas 
\par \textbf{Subcase 2.2.} Let \bea\label{e3.12}ie^{\alpha_1(z+c)-\alpha_3(z)}=1\;\;\text{and}\;\; ie^{\alpha_4(z+c)-\alpha_3(z)}=1.\eea
\par Therefore, it follows from \eqref{e3.5} and \eqref{e3.12} that \bea\label{e3.13} ie^{\alpha_2(z+c)-\alpha_4(z)}=-1\;\;\text{and}\;\;ie^{\alpha_3(z+c)-\alpha_2(z)}=-1.\eea
\par As all $\alpha_j$'s are polynomials in $\mathbb{C}^n$, it follows from \eqref{e3.12} and \eqref{e3.13} that 
$\alpha_1(z+c)-\alpha_3(z)=\zeta_1$, $\alpha_4(z+c)-\alpha_3(z)=\zeta_2$, $\alpha_2(z+c)-\alpha_4(z)=\zeta_3$ and $\alpha_3(z+c)-\alpha_2(z)=\zeta_4$, $\zeta_j\in\mathbb{C}$ for $j=1,2,3,4$. This implies that $\alpha_1(z+2c)-\alpha_2(z)=\zeta_1+\zeta_4$ and $\alpha_2(z+2c)-\alpha_1(z)=\zeta_2+\zeta_3$. Thus, $\alpha_1(z+4c)-\alpha_1(z)=\alpha_2(z+4c)-\alpha_2(z)=\sum_{j=1}^{4}\zeta_j$. Hence, we can get that $\alpha_1(z)=L(z)+\Phi(z)+k_1$ and $\alpha_2(z)=L(z)+\Phi(z)+k_2$, where $L(z)=\sum_{j=1}^{n}a_jz_j$, $\Phi(z)$ is defined as in \eqref{e2.5}, $a_j, k_1,k_2\in\mathbb{C}$ for $j=1,2,\ldots,n$. But, the we must obtain that $\alpha_2(z)-\alpha_1(z)=k_2-k_1\in\mathbb{C}$, a contradiction.\vspace{1.2mm}
\par \textbf{Subcase 2.3.}  Let \beas ie^{\alpha_2(z+c)-\alpha_3(z)}=1\;\;\text{and}\;\; ie^{\alpha_3(z+c)-\alpha_3(z)}=1.\eeas
\par By similar argument as used in \textbf{Subcase 2.2}, we can easily get a contradiction.\vspace{1.2mm}
\par \textbf{Subcase 2.4.}  Let \bea\label{e3.14} ie^{\alpha_2(z+c)-\alpha_3(z)}=1\;\; ie^{\alpha_4(z+c)-\alpha_1(z)}=1.\eea
\par In view of \eqref{e3.5} and \eqref{e3.14}, it follows that \bea\label{e3.15} ie^{\alpha_1(z+c)-\alpha_4(z)}=-1\;\;\text{and}\;\; ie^{\alpha_3(z+c)-\alpha_2(z)}=-1.\eea
\par Since $\alpha_1, \alpha_2,\alpha_3,\alpha_4$ are all polynomials in $\mathbb{C}^n$, it follows from \eqref{e3.14} and \eqref{e3.15} that $\alpha_2(z+c)-\alpha_3(z)=l_1$, $\alpha_4(z+c)-\alpha_1(z)=l_2$, $\alpha_1(z+c)-\alpha_4(z)=l_3$ and $\alpha_3(z+c)-\alpha_2(z)=l_4$, $l_j\in\mathbb{C}$ for $j=1,2,3,4$. This implies that $\alpha_2(z+2c)-\alpha_2(z)=\alpha_3(z+2c)-\alpha_3(z)=l_1+l_4$ and $\alpha_4(z+2c)-\alpha_4(z)=\alpha_1(z+2c)-\alpha_1(z)=l_2+l_3$. Therefore, we must have $\alpha_2(z)=L_1(z)+\Phi(z)+B_1$, $\alpha_3(z)=L_1(z)+\Phi(z)+B_2$, $\alpha_1(z)=L_2(z)+\Psi(z)+B_3$ and $\alpha_4(z)=L_2(z)+\Psi(z)+B_4$, where $L_i(z)=\sum_{j=1}^{n}a_{ij}z_j$, $i=1,2$ and $\Phi(z),\Psi(z)$ are polynomials in $\mathbb{C}^n$ defined as in \eqref{e2.5} and \eqref{e2.6}, respectively, $a_{1j},a_{2j}, B_i\in\mathbb{C}$ for $j=1,2,\ldots,n$ and $i=1,\ldots,4$. Thus, $g_1(z)=L(z)+R(z)+D_1$ and $g_2(z)=L(z)+R(z)+D_2$, where $L(z)=L_1(z)+L_2(z)$, $R(z)=\Phi(z)+\Psi(z)$, $D_1=B_1+B_3$ and $D_2=B_2+B_4$.\vspace{1.2mm}
\par Hence, from \eqref{e3.14} and \eqref{e3.15}, we get \bea\label{e3.16} \begin{cases}		ie^{L_1(c)+B_1-B_2}=1,\;\;ie^{L_2(c)+B_4-B_3}=1,\\ie^{L_2(c)+B_3-B_4}=-1\;\;ie^{L_1(c)+B_2-B_1}=-1.\end{cases}\eea
\par After simple computation, it follows from \eqref{e3.16} that \beas e^{2L_1(c)}=1,\;\;e^{2L_2(c)}=1,\;\;e^{2(B_1-B_2)}=-1\;\;\text{and}\;\; e^{2(B_3-B_4)}=-1.\eeas
\par Therefore, form the above relation, we get $e^{L_1(c)}=\pm 1$, $e^{L_2(c)}=\pm 1$, $e^{B_1-B_2}=\pm i$ and $e^{B_3-B_4}=\pm i$. Therefore, in view of \eqref{e3.11}, we obtain the following four possible  relations.\vspace{1.2mm}
\begin{enumerate}
	\item [\textbf{(a)}] $e^{L_1(c)}=1$, $e^{L_2(c)}=1$, $e^{B_1-B_2}=-i$, $e^{B_3-B_4}=i$.\vspace{1.2mm}
	\item [\textbf{(b)}] $e^{L_1(c)}=1$, $e^{L_2(c)}=-1$, $e^{B_1-B_2}=-i$, $e^{B_3-B_4}=-i$.\vspace{1.2mm}
	\item [\textbf{(c)}] $e^{L_1(c)}=-1$, $e^{L_2(c)}=1$, $e^{B_1-B_2}=i$, $e^{B_3-B_4}=i$.\vspace{1.2mm}
	\item [\textbf{(d)}] $e^{L_1(c)}=-1$, $e^{L_2(c)}=-1$, $e^{B_1-B_2}=i$, $e^{B_3-B_4}=-i$.\vspace{1.2mm}
\end{enumerate}
\par Hence, it follows from \eqref{e3.4} that \beas\begin{cases} f_1(z)=\frac{1}{2}\left[e^{L_1(z)+\Psi(z)+B_3}+e^{L_2(z)+\Phi(z)+B_1}\right],\vspace{1mm}\\ f_2(z)=\frac{1}{2}\left[e^{L_1(z)+\Phi(z)+B_2}+e^{L_2(z)+\Psi(z)+B_4}\right].\end{cases}\eeas 
	\par \textbf{Case 3.} Let $\alpha_4-\alpha_3=\eta\in\mathbb{C}$ and $\alpha_2-\alpha_1$ be non-constant. Then from \eqref{e3.3}, we see that  $p_2(z)$ is a constant in $\mathbb{C}$.\vspace{1.2mm}
	\par Since $\alpha_2-\alpha_1$ is non-constant, by Lemma \ref{lem3.1a} and from second equation of \eqref{e3.5}, we have either $ie^{\alpha_3(z+c)-\alpha_1(z)}=1$ or $ie^{\alpha_4(z+c)-\alpha_1(z)}=1$.\vspace{1.2mm}
	\par First suppose that $ie^{\alpha_3(z+c)-\alpha_1(z)}=1$. Then, from the second equation of \eqref{e3.5}, we have $ie^{\alpha_4(z+c)-\alpha_2(z)}=-1$. As all $\alpha_j$, $j=1,2,3,4$ are polynomials in $\mathbb{C}^n$, it follows that $\alpha_3(z+c)-\alpha_1(z)=m_1$ and $\alpha_4(z+c)-\alpha_2(z)=m_2$, where $m_1,m_2\in\mathbb{C}$. Therefore, from \eqref{e3.3}, it follows that $\alpha_2(z)-\alpha_1(z)=-2h_1(z)=m_1-m_2-2h_2$, a constant in $\mathbb{C}$, which contradicts to our assumption.\vspace{1.2mm}
	\par Similarly, we can get a contradiction for the case $ie^{\alpha_4(z+c)-\alpha_1(z)}=1$.\vspace{1.2mm}
	\par \textbf{Case 4.} Let $\alpha_2-\alpha_1=\eta_1\in\mathbb{C}$ and $\alpha_4-\alpha_3$ be non-constant. Then, by similar argument as used in \textbf{Case 3}, we easily get a contradiction.
\end{proof}
\begin{proof}[\textbf{Proof of Theorem $\ref{t2}$}]
Let $(f_1,f_2)$ is a pair of finite transcendental entire solutions of the system \eqref{e2.3}.\vspace{1.2mm}
\par Then, in a similar manner as in Theorem \ref{t1}, we obtain that \bea\label{e3.17} \begin{cases}
\dfrac{\partial^k f_1}{\partial z _1^k}=\displaystyle\frac{1}{2}\left[e^{\alpha_1(z)}+e^{\alpha_2(z)}\right]\\ f_2(z+c)=\displaystyle\frac{1}{2i}\left[e^{\alpha_1(z)}-e^{\alpha_2(z)}\right]\\\dfrac{\partial^k f_2}{\partial z_1^k}=\displaystyle\frac{1}{2}\left[e^{\alpha_3(z)}+e^{\alpha_4(z)}\right]\\f_1(z+c)=\displaystyle\frac{1}{2i}\left[e^{\alpha_3(z)}-e^{\alpha_4(z)}\right],\end{cases}\eea where $\alpha_1, \alpha_2,\alpha_3,\alpha_4$ are defined in \eqref{e3.3}.	
\par After simple computations, it follows from \eqref{e3.17} that \bea\label{e3.18} \begin{cases}
ie^{\alpha_1(z+c)-\alpha_3(z)}+ie^{\alpha_2(z+c)-\alpha_3(z)}+p_2(z)e^{\alpha_4(z)-\alpha_3(z)}=p_1(z),\\ie^{\alpha_3(z+c)-\alpha_1(z)}+ie^{\alpha_4(z+c)-\alpha_1(z)}+p_4(z)e^{\alpha_2(z)-\alpha_1(z)}=p_3(z),\end{cases}\eea with \beas \begin{cases}
p_1(z)=\left(\frac{\partial \alpha_3}{\partial z_1}\right)^k+M_k\left(\frac{\partial^k \alpha_3}{\partial z_1^k},\ldots, \frac{\partial \alpha_3}{\partial z_1}\right),\\ p_2(z)=\left(\frac{\partial \alpha_4}{\partial z_1}\right)^k+N_k\left(\frac{\partial^k \alpha_4}{\partial z_1^k},\ldots, \frac{\partial \alpha_4}{\partial z_1}\right)\\p_3(z)=\left(\frac{\partial \alpha_1}{\partial z_1}\right)^k+O_k\left(\frac{\partial^k \alpha_1}{\partial z_1^k},\ldots, \frac{\partial \alpha_1}{\partial z_1}\right),\\p_4(z)=\left(\frac{\partial \alpha_2}{\partial z_1}\right)^k+R_k\left(\frac{\partial^k \alpha_2}{\partial z_1^k},\ldots, \frac{\partial \alpha_2}{\partial z_1}\right),\end{cases}\eeas where $M_k$ is the partial differential polynomial in $ \frac{\partial \alpha_3}{\partial z_1}$ of order less than $k$ in which $ \frac{\partial \alpha_3}{\partial z_1}$ appears in product with at least one higher order partial derivatives  with respect to $z_1$ of $\alpha_3$. Similar definitions for $N_k, O_k$ and $R_k$.
\par Now, we consider the following four possible cases.\vspace{1.2mm}
\par \textbf{Case 1.} Let $\alpha_2(z)-\alpha_1(z)=\eta_1$ and $\alpha_4(z)-\alpha_3(z)=\eta_2$, where $\eta_1,\eta_2\in\mathbb{C}$. Then, in view of \eqref{e3.3}, it follows that $h_1(z)$ and $h_2(z)$ are both constants in $\mathbb{C}$. Set $e^{h_1}=\xi_1$ and $e^{h_2}=\xi_2$, $\xi_1,\xi_2$ are non-zero constants in $\mathbb{C}$. Therefore, in view of \eqref{e3.17}, we obtain that \bea\label{e3.19} \begin{cases}
\dfrac{\partial^k f_1}{\partial z_1^k}=\dfrac{\xi_1+\xi_1^{-1}}{2}e^{\frac{1}{2}g_1(z)},\\ f_2(z+c)=\dfrac{\xi_1-\xi_1^{-1}}{2i}e^{\frac{1}{2}g_1(z)},\\\dfrac{\partial^k f_2}{\partial z_1^k}=\dfrac{\xi_2+\xi_2^{-1}}{2}e^{\frac{1}{2}g_2(z)},\\f_1(z+c)=\dfrac{\xi_2-\xi_2^{-1}}{2i}e^{\frac{1}{2}g_2(z)}.\end{cases}\eea
\par It is clear from \eqref{e3.19} that $\xi_1^4,\xi_2^4\neq1$. Also, after simple computations, we obtain from \eqref{e3.19} that \bea\label{e3.20} \begin{cases}
e^{\frac{1}{2}[g_1(z+c)-g_2(z)]}=\dfrac{\xi_2-\xi_2^{-1}}{2i(\xi_1+\xi_1^{-1})}q_1(z),\\e^{\frac{1}{2}[g_2(z+c)-g_1(z)]}=\dfrac{\xi_1-\xi_1^{-1}}{2i(\xi_2+\xi_2^{-1})}q_2(z),\end{cases}\eea with \beas q_1(z)=\left(\frac{1}{2}\frac{\partial g_2}{\partial z_1}\right)^k+T_1\left(\frac{\partial^kg_2}{\partial z_1^k},\ldots, \frac{\partial g_2}{\partial z_1}\right),\;q_2(z)=\left(\frac{1}{2}\frac{\partial g_1}{\partial z_1}\right)^k+T_2\left(\frac{\partial^kg_1}{\partial z_1^k},\ldots, \frac{\partial g_1}{\partial z_1}\right),\eeas where $T_1$ is a partial differential polynomial in $\frac{\partial g_2}{\partial z_1}$ of order less than $k$ where $\frac{\partial g_2}{\partial z_1}$ appears in the product with at least one more partial derivative of higher dimension, and similar definition for $T_2$.\vspace{1.2mm}
\par Since $g_1(z),g_2(z)$ are polynomials in $\mathbb{C}^2$, it follows from \eqref{e3.20} that $q_1,q_2, g_2(z+c)-g_1(z)$ and $g_2(z+c)-g_1(z)$ are all constants in $\mathbb{C}$. Let $g_1(z+c)-g_2(z)=\eta_3$ and $g_2(z+c)-g_1(z)=\eta_4$, $\eta_3,\eta_4\in\mathbb{C}$. This implies that $g_1(z+2c)-g_1(z)=g_2(z+2c)-g_2(z)=\eta_3+\eta_4$.\vspace{1.2mm}
\par Thus, we have \beas g_1(z)=L(z)+H(s)+B_1,\;\;\text{and}\;\;g_2(z)=L(z)+H(s)+B_2,\eeas where $L(z)=a_1z_1+a_2z_2$, $H(s)$ is a polynomial in $s:=d_1z_1+d_2z_2$ with $d_1c_1+d_2c_2=0$, $a_1,a_2,B_1,B_2,d_1,d_2\in\mathbb{C}$. Also from the form of $q_1(z)$ and $q_2(z)$, we conclude that $\frac{\partial g_1}{\partial z_1}$ and $\frac{\partial g_2}{\partial z_1}$ both are constants. Therefore, we must have $s:=d_2z_2$ with $d_2c_2=0$.\vspace{1.2mm}
\par Therefore, it follows from \eqref{e3.20} that \bea\label{e3.21} \begin{cases}
e^{\frac{1}{2}[L(c)+B_1-B_2]}=\left(\dfrac{a_1}{2}\right)^k\dfrac{\xi_2-\xi_2^{-1}}{i(\xi_1+\xi_1^{-1})},\\e^{\frac{1}{2}[L(c)+B_2-B_1]}=\left(\dfrac{a_1}{2}\right)^k\dfrac{\xi_1-\xi_1^{-1}}{i(\xi_2+\xi_2^{-1})}.\end{cases}\eea
\par Hence, from we get \eqref{e3.21} that \beas e^{L(c)}=-\left(\dfrac{a_1}{2}\right)^{2k}\dfrac{(\xi_1-\xi_1^{-1})(\xi_2-\xi_2^{-1})}{(\xi_1+\xi_1^{-1})(\xi_2+\xi_2^{-1})}\;\;\text{and}\;\;e^{B_1-B_2}=\dfrac{\xi_2^2-\xi_2^{-2}}{\xi_1^2-\xi_1^{-2}}.\eeas
\par Thus, from \eqref{e3.19}, we get \beas\begin{cases}
f_1(z_1,z_2)=\dfrac{\xi_2^2-1}{2i\xi_2}e^{\frac{1}{2}[L(z)+H(s)-L(c)+B_2]},\\ f_2(z_1,z_2)=\dfrac{\xi_1^2-1}{2i\xi_1}e^{\frac{1}{2}[L(z)+H(s)-L(c)+B_1]}.\end{cases}\eeas
\par \textbf{Case 2.} Let $\alpha_2(z)-\alpha_1(z)$ and $\alpha_4(z)-\alpha_3(z)$ both are non-constant.\vspace{1.2mm}
\par Observe that $p_1(z)$ and $p_2(z)$ both can not be simultaneously zero. Otherwise, from the first equation of \eqref{e3.18}, we obtain $e^{\alpha_1(z+c)-\alpha_2(z+c)}=-1$, which is a contradiction as $\alpha_2(z)-\alpha_1(z)$ is non-constant.\vspace{1.2mm}
\par Similarly, in view of the second equation of \eqref{e3.18}, we conclude that $p_3(z)$ and $p_4(z)$ can not be simultaneously zero.\vspace{1.2mm}
\par Now, let $p_1(z)\not\equiv0$ and $p_2(z)\equiv0$. Then, first equation of \eqref{e3.18} reduces to \bea\label{e3.23} e^{\alpha_1(z+c)-\alpha_3(z)}+e^{\alpha_2(z+c)-\alpha_3(z)}=-ip_1(z).\eea
\par As $\alpha_2(z)-\alpha_1(z)$ is a non-constant polynomial in $\mathbb{C}^2$, it follows from \eqref{e3.23} that $\alpha_1(z+c)-\alpha_3(z)$ and $\alpha_2(z+c)-\alpha_3(z)$ both are non-constants.\vspace{1.2mm}
\par Now, rewrite \eqref{e3.23} as \bea\label{e3.24} e^{\alpha_1(z+c)}+e^{\alpha_2(z+c)}+ip_1(z)e^{\alpha_3(z)}\equiv0.\eea
\par Therefore, in view of Lemma \ref{lem3.7}, we can easily get a contradiction from \eqref{e3.24}.\vspace{1.2mm}
\par In a similar manner we can get a contradiction when  $p_1(z)\equiv0$ and $p_2(z)\not\equiv0$. Hence, $p_1(z)\not\equiv0$ and $p_2(z)\not\equiv0$. By similar arguments, we also get $p_3(z)\equiv0$ and $p_4(z)\not\equiv0$.\vspace{1.2mm}
\par Therefore, by \eqref{e3.18} and Lemma \ref{lem3.1a}, we obtain that \beas \text{either}\;\; ie^{\alpha_1(z+c)-\alpha_3(z)}=p_1(z)\;\; \text{or}\;\;ie^{\alpha_2(z+c)-\alpha_3(z)}=p_1(z)\eeas and \beas \text{either}\;\; ie^{\alpha_3(z+c)-\alpha_1(z)}=p_3(z)\;\; \text{or}\;\;ie^{\alpha_4(z+c)-\alpha_1(z)}=p_3(z).\eeas
\par Now, we consider the four possible subcases.\vspace{1.2mm}
\par \textbf{Subcase 2.1.} Let \bea\label{e3.25}
ie^{\alpha_1(z+c)-\alpha_3(z)}=p_1(z)\;\;\text{and}\;\;ie^{\alpha_3(z+c)-\alpha_1(z)}=p_3(z).\eea
Then, by \eqref{e3.18} and \eqref{e3.25}, we get 
\bea\label{e3.26}-ie^{\alpha_2(z+c)-\alpha_4(z)}=p_2(z)\;\;\text{and}\;\;-ie^{\alpha_4(z+c)-\alpha_2(z)}=p_4(z).\eea
\par Since all $\alpha_j$'s are polynomials in $\mathbb{C}^2$, $j=1,\ldots,4$, from \eqref{e3.25} and \eqref{e3.26}, we conclude that $p_1,p_2,p_3,p_4$, $\alpha_1(z+c)-\alpha_3(z),\alpha_3(z+c)-\alpha_1(z),\alpha_2(z+c)-\alpha_4(z)=k_3$ and $\alpha_4(z+c)-\alpha_2(z)=k_4$ are all constants in $\mathbb{C}$. Let  $\alpha_1(z+c)-\alpha_3(z)=k_1$, $\alpha_3(z+c)-\alpha_1(z)=k_2$, $\alpha_2(z+c)-\alpha_4(z)=k_3$ and $\alpha_4(z+c)-\alpha_2(z)=k_4$, where $k_j\in\mathbb{C}$, $j=1,\ldots,4$. These imply $\alpha_1(z+2c)-\alpha_1(z)=\alpha_3(z+2c)-\alpha_3(z)=k_1+k_2$ and $\alpha_2(z+2c)-\alpha_2(z)=\alpha_4(z+2c)-\alpha_4(z)=k_3+k_4$. Thus, \beas \alpha_1(z)=L_1(z)+H_1(s)+B_1,\; \alpha_3(z)=L_1(z)+H_1(s)+B_2,\eeas \beas\alpha_2(z)=L_2(z)+H_2(s)+B_3,\; \alpha_4(z)=L_2(z)+H_2(s)+B_4,\eeas where $L_i(z)=a_{i1}z_1+a_{i2}z_2$, $H_{j}(s)$ is a polynomial in $s:=d_1z_1+d_2z_2$ with $d_1c_1+d_2c_2=0$, $i=1,2$, $j=1,2$, $a_{i1},a_{i2},B_1,B_2,B_3,B_4\in\mathbb{C}$. As all $p_j$'s are constants, it follows that $s:=d_2z_2$ with $d_2c_2=0$.\vspace{1.2mm}
\par Since $\alpha_2(z)-\alpha_1(z)$ and $\alpha_4(z)-\alpha_3(z)$ are non-constants, we conclude that $L_1(z)+H_1(s)\neq L_2(z)+H_2(s)$. Thus, in view of \eqref{e3.3}, we have \beas g_1(z)=L_1(z)+L_2(z)+H_1(s)+H_2(s)+B_1+B_3\;\;\text{and}\eeas \beas g_2(z)=L_1(z)+L_2(z)+H_1(s)+H_2(s)+B_2+B_4.\eeas
\par Therefore, in view of \eqref{e3.25} and \eqref{e3.26}, we obtain \bea\label{e3.27}\begin{cases} ie^{L_1(c)+B_1-B_2}=a_{11}^k,\;\;ie^{L_1(c)+B_2-B_1}=a_{11}^k,\\ie^{L_2(c)+B_3-B_4}=-a_{21}^k,\;\;ie^{L_2(c)+B_4-B_3}=-a_{21}^k.\end{cases}\eea
\par From first two equations of \eqref{e3.27}, we obtain $e^{2L_1(c)}=-a_{11}^{2k}\;\;\text{and}\;\;e^{2(B_1-B_2)}=1$ and from last two equations of \eqref{e3.27}, we get $e^{2L_2(c)}=-a_{21}^{2k}\;\;\text{and}\;\;e^{2(B_3-B_4)}=1$. These imply $e^{L_1(c)}=\pm ia_{11}^{k}$, $e^{L_2(c)}=\pm ia_{21}^{k}$, $e^{(B_1-B_2)}=\pm 1$ and $e^{(B_3-B_4)}=\pm 1$.\vspace{1.2mm}
\par If $e^{L_1(c)}=ia_{11}^k$, then from \eqref{e3.27}, we get $e^{B_1-B_2}=-1$ and if $e^{L_1(c)}=-ia_{11}^k$, then from \eqref{e3.27}, we obtain $e^{B_1-B_2}=1$. Similarly, if $e^{L_2(c)}=ia_{21}^k$, then from \eqref{e3.27}, we get $e^{B_3-B_4}=1$ and if $e^{L_1(c)}=-ia_{11}^k$, then from \eqref{e3.27}, we obtain $e^{B_3-B_4}=-1$.\vspace{1.2mm}
\par Thus, from fourth and second equations of \eqref{e3.17}, we obtain \beas f_1(z)=\frac{1}{2i}\left[e^{L_1(z)+H_1(s)-L_1(c)+B_2}-e^{L_2(z)+H_2(s)-L_2(c)+B_4}\right]\eeas and \beas f_2(z)=\frac{1}{2i}\left[e^{L_1(z)+H_1(s)-L_1(c)+B_1}-e^{L_2(z)+H_2(s)-L_2(c)+B_3}\right].\eeas
\par \textbf{Subcase 2.2.} Let \bea\label{e3.29}ie^{\alpha_1(z+c)-\alpha_3(z)}=p_1(z)\;\;\text{and}\;\;ie^{\alpha_4(z+c)-\alpha_1(z)}=p_3(z).\eea
\par Then, in view of \eqref{e3.18} and \eqref{e3.29}, we obtain \bea\label{e3.30}-ie^{\alpha_2(z+c)-\alpha_4(z)}=p_2(z)\;\;\text{and}\;\;-ie^{\alpha_3(z+c)-\alpha_2(z)}=p_4(z).\eea
\par Since $\alpha_1(z),\alpha_2(z),\alpha_3(z), \alpha_4(z)$ are polynomials in $\mathbb{C}^2$, we conclude from \eqref{e3.29} and \eqref{e3.30} that $\alpha_1(z+c)-\alpha_3(z)=m_1$, $\alpha_4(z+c)-\alpha_1(z)=m_2$, $\alpha_2(z+c)-\alpha_4(z)=m_3$ and $\alpha_3(z+c)-\alpha_2(z)=m_4$, where $m_j\in\mathbb{C}$, $j=1,\ldots,4$. These imply $\alpha_1(z+2c)-\alpha_1(z)=\alpha_2(z+4c)-\alpha_2(z)=\sum_{j=1}^{4}m_j$. Therefore, $\alpha_1(z)=L(z)+H(s)+B_1$ and $\alpha_2(z)=L(z)+H(s)+B_2$, where $L(z)=a_1z_1+a_2z_2$ and $H(s)$ is a polynomial in $s:=d_1z_1+d_2z_2$ with $d_1c_1+d_2c_2=0$, $a_1,a_2,B_1,B_2\in\mathbb{C}$. Hence $\alpha_2(z)-\alpha_1(z)=B_2-B_1\in\mathbb{C}$, a contradiction.\vspace{1.2mm}
\par \textbf{Subcase 2.3.} Let \beas
ie^{\alpha_2(z+c)-\alpha_3(z)}=p_1(z)\;\;\text{and}\;\;ie^{\alpha_3(z+c)-\alpha_1(z)}=p_3(z).\eeas
\par Then, by similar arguments as in Subcase 2.2, we easily get a contradiction.\vspace{1.2mm}
\par \textbf{Subcase 2.4.} Let \bea\label{e3.31}	ie^{\alpha_2(z+c)-\alpha_3(z)}=p_1(z)\;\;\text{and}\;\;ie^{\alpha_4(z+c)-\alpha_1(z)}=p_3(z).\eea
\par Then, by \eqref{e3.18} and \eqref{e3.25}, we get \bea\label{e3.32}-ie^{\alpha_1(z+c)-\alpha_4(z)}=p_2(z)\;\;\text{and}\;\;-ie^{\alpha_3(z+c)-\alpha_2(z)}=p_4(z).\eea
\par As $\alpha_1(z),\alpha_2(z),\alpha_3(z), \alpha_4(z)$ are polynomials in $\mathbb{C}^2$, from \eqref{e3.31} and \eqref{e3.32}, we conclude that $p_1,p_2,p_3,p_4$, $\alpha_2(z+c)-\alpha_3(z)$, $\alpha_4(z+c)-\alpha_1(z)$, $\alpha_1(z+c)-\alpha_4(z)$ and $\alpha_3(z+c)-\alpha_2(z)$ are all constants in $\mathbb{C}$.\vspace{1.2mm}
\par Let $\alpha_2(z+c)-\alpha_3(z)=n_1$, $\alpha_4(z+c)-\alpha_1(z)=n_2$, $\alpha_1(z+c)-\alpha_4(z)=n_3$ and $\alpha_3(z+c)-\alpha_2(z)=n_4$, where $n_j\in\mathbb{C}$, $j=1,\ldots,4$. Thus, we get $\alpha_1(z+2c)-\alpha_1(z)=\alpha_4(z+2c)-\alpha_4(z)=n_2+n_3$ and $\alpha_2(z+2c)-\alpha_2(z)=\alpha_3(z+2c)-\alpha_3(z)=n_1+n_4$. Hence, we assume \beas \alpha_1(z)=L_1(z)+H_1(s)+B_1,\;\; \alpha_4(z)=L_1(z)+H_1(s)+B_2,\eeas \beas \alpha_2(z)=L_2(z)+H_2(s)+B_3\;\;\text{and}\;\; \alpha_3(z)=L_2(z)+H_2(s)+B_4,\eeas where $L_i(z)=a_{i1}z_1+a_{i2}z_2$, $H_{j}(s)$ is a polynomial in $s:=d_1z_1+d_2z_2$ with $d_1c_1+d_2c_2=0$ for $i=1,2$, $j=1,2$, $a_{i1},a_{i2},B_1,B_2,B_3,B_4\in\mathbb{C}$. Since $\alpha_2(z)-\alpha_1(z)$ and $\alpha_4(z)-\alpha_3(z)$ are non-constants, it follows that $L_1(z)+H_1(s)\neq L_2(z)+H_2(s)$. Also since $p_1,p_2,p_3,p_4$ are all constants, it follows that $\frac{\partial \alpha_j}{\partial z_1}$ are constants for all j=1,2,3,4. Hence, $H_1$ and $H_2$ are the polynomials in $s:=d_2z_2$ with $d_2c_2=0$.\vspace{1.2mm}
\par Now, in view of \eqref{e3.3}, we get \beas\begin{cases}
g_1(z)=L_1(z)+L_2(z)+H_1(s)+H_2(s)+B_1+B_3,\\ g_2(z)=L_1(z)+L_2(z)+H_1(s)+H_2(s)+B_2+B_4.
\end{cases}\eeas
\par Therefore, in view of \eqref{e3.31} and \eqref{e3.32}, we obtain \bea\label{e3.33}\begin{cases} ie^{L_2(c)+B_3-B_4}=a_{21}^{k},\;\;ie^{L_1(c)+B_2-B_1}=a_{11}^k,\\ie^{L_1(c)+B_1-B_2}=-a_{11}^k,\;\;ie^{L_2(c)+B_4-B_3}=-a_{21}^k.\end{cases}\eea
\par Thus, from \eqref{e3.33}, we obtain \beas e^{2L_1(c)}=a_{11}^{2k},\;\;e^{2L_2(c)}=a_{21}^{2k},\;\;e^{2(B_1-B_2)}=-1,\;\;\text{and}\;\;e^{2(B_3-B_4)}=-1,\eeas which imply \beas e^{L_1(c)}=\pm a_{11}^{k},\;\;e^{L_2(c)}=\pm a_{21}^{k},\;\;e^{B_1-B_2}=\pm i,\;\;\text{and}\;\;e^{B_3-B_4}=\pm i.\eeas
\par If $e^{L_1(c)}=a_{11}^k$, then from \eqref{e3.33}, we have $e^{B_1-B_2}=i$ and if $e^{L_1(c)}=-a_{11}^k$, then we have $e^{B_1-B_2}=-i$. Similarly, if $e^{L_2(c)}=a_{21}^k$, the from \eqref{e3.33}, we have $e^{B_3-B_4}=-i$ and if $e^{L_2(c)}=-a_{21}^k$, the from \eqref{e3.33}, we have $e^{B_3-B_4}=i$.\vspace{1.2mm}
\par Hence, from fourth and second equations of \eqref{e3.17}, we get
\beas\begin{cases}
f_1(z)=-\frac{1}{2i}\left[e^{L_1(z)+H_1(z)-L_1(c)+B_2}-e^{L_2(z)+H_2(z)-L_2(c)+B_4}\right]\\f_2(z)\;=\frac{1}{2i}\left[e^{L_1(z)+H_1(z)-L_1(c)+B_1}-e^{L_2(z)+H_2(z)-L_2(c)+B_3}\right].\end{cases}\eeas 
\par \textbf{Case 3.} Let $\alpha_4-\alpha_3=\eta\in\mathbb{C}$ and $\alpha_2-\alpha_1$ be non-constant. Then, \eqref{e3.3} yields that $p_2(z)$ is a constant in $\mathbb{C}$.\vspace{1.2mm}
\par Then first equation of \eqref{e3.18} yields that \bea\label{e3.32a} ie^{\alpha_1(z+c)-\alpha_3(z)}+ie^{\alpha_2(z+c)-\alpha_3(z)}=p_1(z)-e^{\eta}p_2(z).\eea
\par From \eqref{e3.32a}, we observe that $T\left(r,e^{\alpha_1(z+c)-\alpha_3(z)}\right)=T\left(r,e^{\alpha_2(z+c)-\alpha_3(z)}\right)+O(1)$.\vspace{1.2mm}
\par We also note that \beas N\left(r,e^{\alpha_1(z+c)-\alpha_3(z)}\right)=N\left(r,\frac{1}{e^{\alpha_1(z+c)-\alpha_3(z)}}\right)=S\left(r,e^{\alpha_1(z+c)-\alpha_3(z)}\right)\eeas and \beas N\left(r,\frac{1}{e^{\alpha_1(z+c)-\alpha_3(z)}-w}\right)=N\left(r,\frac{1}{e^{\alpha_2(z+c)-\alpha_3(z)}}\right)=S\left(r,e^{\alpha_2(z+c)-\alpha_3(z)}\right),\eeas where $w=-i(p_1-e^{eta}p_2)$.\vspace{1.2mm}
Now, by the second fundamental theorem of Nevanlinna for several complex variables, we obtain \beas && T\left(r,e^{\alpha_1(z+c)-\alpha_3(z)}\right)\\&&\leq \ol N\left(r,e^{\alpha_1(z+c)-\alpha_3(z)}\right)+\ol N\left(r,\frac{1}{e^{\alpha_1(z+c)-\alpha_3(z)}}\right)+\ol N\left(r,\frac{1}{e^{\alpha_1(z+c)-\alpha_3(z)}-w}\right)\\&&+S\left(r,e^{\alpha_1(z+c)-\alpha_3(z)}\right)\\&&\leq S\left(r,e^{\alpha_1(z+c)-\alpha_3(z)}\right)+S\left(r,e^{\alpha_2(z+c)-\alpha_3(z)}\right).\eeas
\par This implies that $\alpha_1(z+c)-\alpha_3(z)$ is a constant in $\mathbb{C}$. Therefore, $\alpha_2(z+c)-\alpha_3(z)$ is also constant. Thus, $\alpha_2(z)-\alpha_1(z)$ becomes constant, which is a contradiction.\vspace{1.2mm}
\par \textbf{Case 4.} Let $\alpha_2-\alpha_1=\eta_1\in\mathbb{C}$ and $\alpha_4-\alpha_3$ be non-constant. Then, by similar argument as used in \textbf{Case 3}, we easily get a contradiction.\end{proof}
\begin{proof}[\textbf{Proof of Theorem $\ref{t3}$}]
Let $(f_1(z),f_2(z))$ be a pair of finite order transcendental entire solution of \eqref{e2.4}.\vspace{1.2mm}
\par Then, by similar argument as in Theorem \ref{t1}, we obtain \bea\label{e2.33} \begin{cases}f_1^{(k)}(z)=\displaystyle\frac{1}{2}\left[e^{p_1(z)}+e^{-p_1(z)}\right],\\ f_2(z+c)-f_2(z)=\displaystyle\frac{1}{2i}\left[e^{p_1(z)}-e^{-p_1(z)}\right],\\f_2^{(k)}(z)=\displaystyle\frac{1}{2}\left[e^{p_2(z)}+e^{-p_2(z)}\right],\\f_1(z+c)-f_1(z)=\displaystyle\frac{1}{2i}\left[e^{p_2(z)}-e^{-p_2(z)}\right],\end{cases}\eea where $p_1(z)$ and $p_2(z)$ are two non-constant polynomials.\vspace{1.2mm}
\par Differentiating fourth equation of \eqref{e2.33} $k$ times, we get \bea\label{e2.35a} f_1^{(k)}(z+c)-f_1^{(k)}(z)=\frac{1}{2i}\left[G_1e^{p_2(z)}-G_2e^{-p_2(z)}\right],\eea where $G_1=\left(p_2^{\prime}\right)^k+M_k\left(p_2^{(k)},\ldots,p_2^{\prime}\right)$ and $G_2=\left(-p_2^{\prime}\right)^k+N_k\left(p_2^{(k)},\ldots,p_2^{\prime}\right)$, where $M_k$ is a differential polynomial in $p_2(z)$ of degree less than $k$ in which $p_2^{\prime}$ appears in the product with at least one higher order derivative of $p_2$. Similar definition is for $N_k$.\vspace{1.2mm}
\par Now in view of the first equation of \eqref{e2.33} and \eqref{e2.35a}, we obtain \bea&&\label{e2.34}	-ie^{p_1(z+c)+p_2(z)}-ie^{-p_1(z+c)+p_2(z)}+ie^{p_1(z)+p_2(z)}+ie^{-p_1(z)+p_2(z)}\nonumber\\&&+G_1e^{2p_2(z)}=G_2.\eea
\par Similarly in view of the second and third equations of \eqref{e2.33}, we obtain \bea&&\label{e2.35}-ie^{p_2(z+c)+p_1(z)}-ie^{-p_2(z+c)+p_1(z)}+ie^{p_2(z)+p_1(z)}+ie^{-p_2(z)+p_1(z)}\nonumber\\&&+G_3e^{2p_1(z)}=G_4,\eea where $G_3=\left(p_1^{\prime}\right)^k+O_k\left(p_1^{(k)},\ldots,p_1^{\prime}\right)$ and $G_4=\left(-p_1^{\prime}\right)^k+R_k\left(p_1^{(k)},\ldots,p_1^{\prime}\right)$, where $O_k$ is a differential polynomial in $p_1(z)$ of degree less than $k$ in which $p_1^{\prime}$ appears in the product with at least one more derivative of $p_1$ of order greater than or equal to $2$. Similar definition is for $R_k$.\vspace{1.2mm}
\par Now, we consider the following two possible cases.\vspace{1.2mm}
\par \textbf{Case 1.} Let $p_2(z)-p_1(z)=\eta$, $\eta\in\mathbb{C}$. Then \eqref{e2.34} and \eqref{e2.35} yield \bea\label{e2.36}\begin{cases}-ie^{\eta}e^{p_1(z+c)+p_1(z)}-ie^{\eta}e^{p_1(z)-p_1(z+c)}+e^{\eta}(i+G_1e^{\eta})e^{2p_1(z)}=G_2-ie^{\eta},\\  -ie^{\eta}e^{p_1(z+c)+p_1(z)}-ie^{-\eta}e^{p_1(z)-p_1(z+c)}+(ie^{\eta}+G_3)e^{2p_1(z)}=G_4-ie^{-\eta}.\end{cases}\eea
\par Note that $i+G_1e^{\eta}$ and $G_2-ie^{\eta}$ can not be zero simultaneously. Otherwise, from the first equation of \eqref{e2.36}, we get $p_1(z+c)$, and hence $p_1(z)$ is constant, a contradiction.\vspace{1.2mm}
\par Next, suppose that $i+G_1e^{\eta}\not\equiv0$ and $G_2-ie^{\eta}\equiv0$. Then the first equation of \eqref{e2.36} reduces to \bea\label{e2.37}ie^{p_1(z+c)-p_1(z)}+ie^{-(p_1(z)+p_1(z+c))}=i+G_1e^{\eta}.\eea
\par Now, observe that \beas N\left(r,e^{-(p_1(z+c)+p_1(z))}\right)=N\left(r,\frac{1}{e^{-(p_1(z+c)+p_1(z))}}\right)=S\left(r,e^{-(p_1(z+c)+p_1(z))}\right)\eeas and in view of \eqref{e2.37}, we have \beas N\left(r,\frac{1}{e^{-(p_1(z+c)+p_1(z))}-w}\right)=N\left(r,\frac{1}{e^{p_1(z+c)-p_1(z)}}\right)=S\left(r,e^{-(p_1(z+c)+p_1(z))}\right),\eeas where $w=1-iG_1e^{\eta}$.\vspace{1.2mm}
\par By the second fundamental theorem of Nevanlinna, we obtain \beas&& T\left(r,e^{-(p_1(z+c)+p_1(z))}\right)\\&&\leq \ol N\left(r,e^{-(p_1(z+c)+p_1(z))}\right)+\ol N\left(r,\frac{1}{e^{-(p_1(z+c)+p_1(z))}}\right)\\&&+ \ol N\left(r,\frac{1}{e^{-(p_1(z+c)+p_1(z))}-w}\right)+S\left(r,e^{-(p_1(z+c)+p_1(z))}\right)\\&&\leq S\left(r,e^{-(p_1(z+c)+p_1(z))}\right)+S\left(r,e^{p_1(z+c)-p_1(z)}\right).\eeas
\par This implies that $p_1(z+c)+p_1(z)$ is constant, and hence $p_1(z)$ is constant, which is a contradiction.\vspace{1.2mm}
\par Similarly, we can get a contradiction for the case $i+G_1e^{\eta}\equiv0$ and $G_2-ie^{\eta}\not\equiv0$. Hence, $i+G_1e^{\eta}\not\equiv0$ and $G_2-ie^{\eta}\not\equiv0$.\vspace{1.2mm}
\par In a similar manner, we can show that $ie^{\eta}+G_3\not\equiv0$ and $G_4-ie^{-\eta}\not\equiv0$.\vspace{1.2mm}
\par Therefore, in view of Lemma \ref{lem3.1} and \eqref{e2.36}, we obtain \bea\label{e2.38}\begin{cases}-ie^{\eta}e^{p_1(z)-p_1(z+c)}=G_2-ie^{\eta},\\-ie^{-\eta}e^{p_1(z)-p_1(z+c)}=G_4-ie^{-\eta}.\end{cases}\eea 
\par By \eqref{e2.36} and \eqref{e2.38}, we get \bea\label{e2.39} \begin{cases} ie^{p_1(z+c)-p_1(z)}=i+G_1e^{\eta},\\ ie^{\eta}e^{p_1(z+c)-p_1(z)}=G_3+ie^{\eta}.\end{cases}\eea
\par Since $p_1(z)$ is a non-constant polynomial in $\mathbb{C}^2$, it follows from \eqref{e2.39} that $p_1(z+c)-p_1(z)$, $G_1,G_2,G_3,G_4$ are all constants in $\mathbb{C}$. Therefore, we must have $p_1(z)=L(z)+H(z_2)+\beta$, where $L(z)=\alpha_1z_1+\alpha_2z_2$, $H(z_2)$ is a polynomial in $z_2$, only, $\alpha_1,\alpha_2,\beta\in\mathbb{C}$. Thus, $p_2(z)=L(z)+H(z_2)+\beta+\eta$.\vspace{1.2mm}
\par Therefore, in view of \eqref{e2.38}, \eqref{e2.39}, we obtain \bea\label{e2.43} \begin{cases}-ie^{\eta}e^{-L(c)}=(-\alpha_1)^k-ie^{\eta},\;\; -ie^{\eta}e^{-L(c)}=(-\alpha_1)^k-ie^{-\eta},\\ ie^{L(c)}=i+\alpha_1^k e^{\eta},\;\; ie^{\eta}e^{L(c)}=\alpha_1^k+ie^{\eta}.\end{cases}\eea
\par Now from the first and second equations of \eqref{e2.43}, we easily obtain $e^{2\eta}=1$. From the first and third equations, we get \bea\label{e2.44} (-1)^ke^{\eta}\alpha_1^{2k}+i\left[(-1)^k-1\right]\alpha_1^{k}=0.\eea
	\par If $k$ is even, the in view of \eqref{e2.44} and the fact that $\alpha_1\neq0$, we easily obtain a contradiction. Hence, $k$ must be odd and therefore, from \eqref{e2.44}, we get \beas \alpha_1^k=-2ie^{-\eta}.\eeas 
	\par Thus from the first equation of \eqref{e2.43}, we obtain \beas e^{L(c)}=\frac{ie^{\eta}}{\alpha_1^k+ie^{\eta}}.\eeas
	\par If $e^{\eta}=1$, then $\alpha_1^k=-2i$, and if $e^{\eta}=-1$, then $\alpha_1^k=2i$.\vspace{1.2mm}
	\par Thus, from \eqref{e2.33}, we obtain \beas f_1(z)=\frac{e^{L(z)+H(z_2)+\beta}-e^{-(L(z)+H(z_2)+\beta)}}{2\alpha_1^k},\eeas\beas f_2(z)=\frac{e^{L(z)+H(z_2)+\beta+\eta}-e^{-(L(z)+H(z_2)+\beta+\eta)}}{2\alpha_1^k}.\eeas 
\par \textbf{Case 2.} Let $p_2(z)-p_1(z)$ be non-constant.\vspace{1.2mm}
\par Now we consider two possible subcases.\vspace{1.2mm}
\par \textbf{Subcase 2.1.} Let $p_2(z)+p_1(z)=\eta$, where $\eta$ is a constant in $\mathbb{C}$.\vspace{1.2mm}
\par Then, \eqref{e2.34} and \eqref{e2.35} yield  \bea\label{e2.40}\begin{cases} e^{p_1(z+c)-p_1(z)}+e^{-(p_1(z+c)+p_1(z))}+(ie^{\eta}G_1-1)e^{-2p_1(z)}=1+ie^{-\eta}G_2,\\ -ie^{\eta}e^{p_1(z)-p_1(z+c)}-ie^{-\eta}e^{p_1(z+c)+p_1(z)}+(G_3+ie^{-\eta})e^{2p_1(z)}=G_4-ie^{\eta}.\end{cases}\eea
\par Now, by similar argument as in Case 1, we can prove that $ie^{\eta}G_1-1$, $1+ie^{-\eta}G_2$, $G_3+ie^{-\eta}$ and $G_4-ie^{\eta}$ are all non-zero.\vspace{1.2mm}
\par Therefore, by Lemma \ref{lem3.1} and \eqref{e2.40}, we have \bea\label{e2.41} 	e^{p_1(z+c)-p_1(z)}=1+ie^{-\eta}G_2\;\;\text{and}\;\;-ie^{\eta}e^{p_1(z)-p_1(z+c)}=G_4-ie^{\eta}.\eea
\par From \eqref{e2.40} and \eqref{e2.41}, we obtain \bea\label{e2.47}e^{-p_1(z+c)+p_1(z)}=1-ie^{\eta}G_1\;\;\text{and}\;\;e^{p_1(z+c)-p_1(z)}=1-ie^{\eta}G_3.\eea
\par Since $p_1(z)$ is a non constant polynomial, it follows from \eqref{e2.41} and \eqref{e2.47} that $G_1$, $G_2$, $G_3$, $G_4$ and $p_1(z+c)-p_1(z)$ are all constants. As $p_1(z+c)-p_1(z)$ is constant, we must have $p_1(z)=L(z)+H(z_2)+\beta$, where $L(z)=\alpha_1z_1+\alpha_2z_2$, $H(z_2)$ is a polynomial in $z_2$, only, $\alpha_1,\alpha_2, \beta\in\mathbb{C}$. Thus, $p_2(z)=-(L(z)+H(z_2)+\beta)+\eta$. Since $G_1$, $G_2$, $G_3$, $G_4$ are all constants, in view of the definitions of them, we see that $G_2=G_3=\alpha_1^k$ and $G_1=G_4=(-\alpha_1)^k$\vspace{1.2mm}
\par Therefore, \eqref{e2.41} and \eqref{e2.47} yield that \bea\label{e2.48}\begin{cases}e^{L(c)}=1+ie^{-\eta}\alpha_1^k,\;\;-ie^{\eta}e^{-L(c)}=(-\alpha_1)^k-ie^{\eta},\\e^{-L(c)}=1-ie^{\eta}(-\alpha_1)^k,\;\;e^{L(c)}=1-ie^{\eta}\alpha_1^k.\end{cases}\eea
\par From the above two equations of \eqref{e2.48}, we easily obtain \bea\label{e2.49} (-1)^kie^{-\eta}\alpha_1^k+1+(-1)^k=0.\eea
\par If $k$ is odd, then in view of \eqref{e2.49}, we can get a contradiction. If $k$ is even, then from \eqref{e2.49}, we get \beas \alpha_1^k=2ie^{\eta}.\eeas
\par From first and fourth equations of \eqref{e2.48}, we obtain \beas e^{2\eta}=-1.\eeas
\par Therefore, from \eqref{e2.33}, we have \beas f_1(z_1,z_2)=\frac{e^{L(z)+H(z_2)+\beta}+e^{-(L(z)+H(z_2)+\beta)}}{2\alpha_1^k},\eeas\beas f_2(z_1,z_2)=\frac{e^{-(L(z)+H(z_2)+\beta)+\eta}+e^{L(z)+H(z_2)+\beta-\eta}}{2\alpha_1^k}.\eeas
	\par \textbf{Subcase 2.2.} Let $p_2(z)+p_1(z)$ be a non-constant polynomial.\vspace{1.2mm}
	\par Now we consider four possible cases below.
	\par \textbf{Subcase 2.2.1.} Let $p_1(z+c)+p_2(z)=\zeta_1$ and $p_2(z+c)+p_1(z)=\zeta_2$, where $\zeta_1,\zeta_2\in\mathbb{C}$. Then, we easily see that $p_1(z+2c)-p_1(z)=\zeta_1-\zeta_2$ and $p_2(z+2c)-p_2(z)=\zeta_2-\zeta_1$. This implies that $p_1(z)=\alpha z+\beta_1$ and $p_2(z)=-\alpha z+\beta_2$, where $\alpha,\beta_1,\beta_2\in\mathbb{C}$. But, then $p_1(z)+p_2(z)=\beta_1+\beta_2=$ constant, a contradiction.\vspace{1.2mm}
	\par \textbf{Subcase 2.2.2.} Let $p_1(z+c)+p_2(z)=\zeta_1\in\mathbb{C}$ and $p_2(z+c)+p_1(z)$ is non-constant.\vspace{1.2mm}
\par Then, by Lemma \ref{lem3.1} and \eqref{e2.35}, we get $-ie^{-p_2(z+c)+p_1(z)}=G_4$. This implies that $-p_2(z+c)+p_1(z)=\zeta_2\in\mathbb{C}$. Therefore, we must have $p_1(z+2c)+p_1(z)=\zeta_1+\zeta_2$, and hence $p_1(z)$ must be a constant, a contradiction.\vspace{1.2mm}
\par \textbf{Subcase 2.2.3.} Let $p_2(z+c)+p_1(z)=\zeta_1\in\mathbb{C}$ and $p_1(z+c)+p_2(z)$ is non-constant.\vspace{1.2mm}
\par Then, by similar argument as in subcase 2.2.2, we obtain a contradiction.\vspace{1.2mm}
\par \textbf{Subcase 2.2.4.} Let $p_2(z+c)+p_1(z)$ and $p_1(z+c)+p_2(z)$ both are non-constant.\vspace{1.2mm}
\par Then, using Lemma \ref{lem3.1}, we obtain from \eqref{e2.34} and \eqref{e2.35} that \beas -ie^{-p_1(z+c)+p_2(z)}=G_2\;\; \text{and}\;\; -ie^{-p_2(z+c)+p_1(z)}=G_4.\eeas
\par As $p_1(z)$ and $p_2(z)$ are non-constant polynomials, we must have $-p_1(z+c)+p_2(z)=\xi_1$ and $-p_2(z+c)+p_1(z)=\xi_2$, where $\xi_1,\xi_2\in\mathbb{C}$. This implies that $-p_1(z+c)+p_1(z)=-p_2(z+c)+p_2(z)=\xi_1+\xi_2$. Therefore, we must get $p_1(z)=\alpha z+\beta_1$ and $p_2(z)=\alpha z+\beta_2$, where $\alpha,\beta_1,\beta_2\in\mathbb{C}$. But, then we get $p_2(z)-p_1(z)=\beta_2-\beta_1=$constant, which is a contradiction.\end{proof}
\section{An open problem} For further investigation, we pose an open problem as follows. What could be the entire solutions of the system of Fermat type differential-difference equation \beas \begin{cases}	\left(\frac{\partial^k f_1}{\partial z_1^k}\right)^2+[f_2(z+c)-f_2(z)]^2=e^{g_1(z_1,z_2)},\\\left(\frac{\partial^k f_2}{\partial z_1^k}\right)^2+[f_1(z+c)-f_1(z)]^2=e^{g_2(z_1,z_2)},\end{cases}\eeas where $g_1(z_1,z_2)$ and $g_2(z_1,z_2)$ are any two polynomials in $\mathbb{C}^2$?

\section{Statements and Declarations}
\vspace{1.3mm}

\noindent \textbf {Conflict of interest} The authors declare that there are no conflicts of interest regarding the publication of this paper.
\vspace{1.5mm}

\noindent{\bf Funding} There is no funding received from any organizations for this research work.
\vspace{1.5mm}

\noindent \textbf {Data availability statement}  Data sharing is not applicable to this article as no database were generated or analyzed during the current study.

\noindent{\bf Acknowledgment:} The authors would like to thank the referee(s) for the helpful suggestions and comments to improve the exposition of the paper.

\end{document}